\documentclass{amsart}
\usepackage{amsfonts}
\usepackage{amsmath}
\usepackage{amssymb}
\usepackage{amsthm}
\numberwithin{equation}{section}

\begin{document}

\title[Preserver problems for the logics associated to Hilbert spaces]
{Preserver problems for the logics associated to Hilbert spaces and related Grassmannians}
\author{Mark Pankov}
\address{Mark Pankov: Faculty of Mathematics and Computer Science, 
University of Warmia and Mazury, Olsztyn, Poland}
\email{pankov@matman.uwm.edu.pl}

\begin{abstract}
We consider the standard quantum logic ${\mathcal L}(H)$ associated to a complex Hilbert space $H$, i.e.
the  lattice of closed subspaces of $H$ together with the orthogonal complementation. 
The orthogonality and compatibility relations are defined for any logic.
In the standard quantum logic, they have a simple interpretation in terms of operator theory.
For example, two closed subspaces (propositions in the logic ${\mathcal L}(H)$) are compatible if and only if 
the projections on these subspaces commute. 
We present both classical and more resent results on transformations of ${\mathcal L}(H)$ and the associated Grassmannians
which preserve the orthogonality or compatibility relation.
The first result in this direction was classical Wigner's theorem.
\end{abstract}

\newtheorem{theorem}{Theorem}[section]
\newtheorem{lemma}{Lemma}[section]
\newtheorem{prop}{Proposition}[section]
\newtheorem{rem}{Remark}[section]
\newtheorem{exmp}{Example}[section]
\newtheorem{prob}{Problem}[section]
\newtheorem{cor}{Corollary}[section]

\maketitle

\section{Introduction}
A {\it logic} is a lattice together with an addition operation known as {\it negation} and satisfying some axioms,
elements of a logic are called {\it propositions}.
By G. Birkhoff and J. Von Neumann \cite{B-VonN}, the logical structure of quantum mechanics corresponds to
the logic whose propositions are subspaces of a finite-dimensional Hilbert space and whose negation is the orthogonal complementation.
Currently,  the {\it standard quantum logic} is the logic formed by closed subspaces of an arbitrary (not necessarily finite-dimensional) complex Hilbert space,
as above, the negation is the orthogonal complementation.
We strongly recommend the short problem book \cite{Cohen} as a quick introduction to the topic and \cite{HB-QL,Pykacz,Var} for more information.
The readers can ask whether interesting logics related to a more general class of vector spaces?
We only note that any logic formed by all closed subspaces of an infinite-dimensional complex Banach space is the standard quantum logic 
(S. Kakutani and G. W. Mackey \cite{KakutaniMackey}).

\subsection{Lattices of closed subspaces}
The lattice consisting of all subspaces of a vector space (not necessarily finite-dimensional) is well-studied.
See, for example, \cite{Baer}.
By the Fundamental Theorem of Projective Geometry \cite[Section III.1, p.44]{Baer},
every automorphism of this lattice is induced by a semilinear automorphism of the corresponding vector space
if the dimension of the vector space is not less than $3$.
For a $2$-dimensional vector space any bijective transformation of the set of $1$-dimensional subspaces gives a lattice automorphism.

Similarly,
every automorphism of the lattice of closed subspaces of an infinite-dimensional normed vector space 
is induced by an invertible bounded linear or conjugate-linear operator
(the second possibility is realized only for the complex case).
For finite-dimensional complex normed spaces this statement fails, since every subspace is closed and 
there are non-bounded semilinear automorphisms associated to non-continuous automorphisms of the field of complex numbers.

In the infinite-dimensional case, 
the description of automorphisms for the lattice of closed subspaces is not a simple consequence of the Fundamental Theorem of Projective Geometry. 
It was first given by G. W. Mackey \cite{Mackey} for real normed spaces.
The complex case easily follows from one technical result obtained by S. Kakutani and G. W. Mackey in \cite{KakutaniMackey}
(this fact is noted in \cite{FillmoreLongstaff}).

The lattice of closed subspaces of a Hilbert space $H$ can be decomposed into the disjoint sum of the following components called {\it Grassmannians}:
\begin{enumerate}
\item[$\bullet$] ${\mathcal G}_{k}(H)$ consisting of all $k$-dimensional subspaces 
and ${\mathcal G}^{k}(H)$ consisting of all closed subspaces of codimension $k$ for every natural $k<\dim H$.
Each ${\mathcal G}^{k}(H)$ can be identified with ${\mathcal G}_{k}(H)$ by the orthogonal complementation.
\item[$\bullet$] ${\mathcal G}_{\infty}(H)$ formed by all closed subspaces whose dimension and codimension both are infinite
(if $H$ is infinite-dimensional).
\end{enumerate}
The Grassmannian ${\mathcal G}_{\infty}(H)$ is partially ordered by the inclusion relation.
Every automorphism of this partially ordered set can be uniquely extended to an automorphism of the lattice of closed subspaces
(M. Pankov \cite{Pankov2, Pankov3}).

\subsection{Automorphisms of the logics associated to Hilbert spaces}
Let $H$ be a complex Hilbert space.
The inner product on $H$ and the associated norm will be denoted by $\langle\cdot,\cdot \rangle$ and $||\cdot ||$, respectively.
We use the standard symbol $\perp$ to denote the orthogonality relation and write $X^{\perp}$ for the orthogonal complement of $X$.

Consider the logic ${\mathcal L}(H)$ formed by closed subspaces of $H$.
Every automorphism of this logic satisfies the following conditions:
\begin{enumerate}
\item[$\bullet$] it is a lattice automorphism,
\item[$\bullet$] it preserves the orthogonal complementation.
\end{enumerate}
The first property implies that such an automorphism is induced by an invertible bounded linear or conjugate-linear operator if $H$ is infinite-dimensional,
and it is defined by a semilinear automorphism if the dimension of $H$ is finite and not less than $3$.
It is not difficult to prove that every semilinear automorphism sending orthogonal vectors to orthogonal vectors is unitary or anti-unitary up to a scalar multiply.
Therefore, the second property guarantees that every automorphism of the logic is induced by an unitary or anti-unitary operator 
if the dimension of $H$ is not less than $3$. 
The statement fails if the dimension is equal to $2$
(consider the set formed by all pairs of orthogonal $1$-dimensional subspaces, 
any bijective transformation of this set induces an automorphism of the logic).

In quantum mechanics, so-called {\it pure states} are identified with $1$-dimensional subspaces of $H$.
The {\it transition probability} between two pure states $P,P'\in {\mathcal G}_{1}(H)$ is equal to
$|\langle x,x' \rangle|$, where $x$ and $x'$ are unit vectors belonging to $P$ and $P'$, respectively.
Classical Wigner's theorem states that every bijective transformation of ${\mathcal G}_{1}(H)$
preserving the transition probability is induced by an unitary or anti-unitary operator.
This is one of basic results of the mathematical foundations of quantum mechanics. 
Note that there is an analogue of this statement for not necessarily bijective transformations \cite{Geher}.
Wigner's theorem can be extended on other Grassmannians.
See \cite{GeherSemrl, Molnar}, where transformations of Grassmannians preserving the principal angles and the gap metric are determined.
Transformations preserving structures related to quantum mechanics are also investigated in \cite[Chapter 2]{Molnar-book}.

Pairs of orthogonal elements from ${\mathcal G}_{1}(H)$ correspond to the case of zero transition probability.
U. Uhlhorn \cite{Uhlhorn} reproved Wigner's theorem in terms of quantum logic as follows.
Every bijective transformation of ${\mathcal G}_{1}(H)$ preserving the orthogonality relation in both directions
is induced by an unitary or anti-unitary operator.
The assumption that the dimension of $H$ is not less than $3$ cannot be omitted.
By M. Gy\"ory \cite{Gyory} and P. {\v S}emrl \cite{Semrl}, 
the same holds for orthogonality preserving (in both directions) bijective transformations of
${\mathcal G}_{\infty}(H)$ and ${\mathcal G}_{k}(H)$ if the dimension of $H$ is greater than $2k$.
In the case when the dimension is equal to $2k$, the statement fails
(as above, we can take any bijective transformation of the set of all pairs of orthogonal $k$-dimensional subspaces).
If the dimension is less than $2k$, then there exist no orthogonal pairs of $k$-dimensional subspaces.

It is not difficult to show that 
every bijective transformation of ${\mathcal G}_{\infty}(H)$ preserving the orthogonality relation in both directions is an automorphism of 
the corresponding partially ordered set.
Therefore, the mentioned above statement concerning orthogonality preserving transformations of ${\mathcal G}_{\infty}(H)$ 
can be obtained from the fact that every automorphism of this partially ordered set is extendable to an automorphism of the lattice of closed subspaces.
The proof of the same statement for ${\mathcal G}_{k}(H)$ will be based on Chow's theorem \cite{Chow}
which describes automorphisms of Grassmann graphs.

\subsection{Compatibility relation}
So-called {\it compatibility relation} is defined for any logic. 
In classical logics, any two propositions are compatible.
Two propositions $X$ and $Y$ in the logic ${\mathcal L}(H)$ are compatible 
if there are propositions $X',Y'$ such that $X\cap Y, X',Y'$ are mutually orthogonal and 
$$X=(X\cap Y)+X',\;\;\;Y=(X\cap Y)+Y'.$$
Any two incident or orthogonal elements of ${\mathcal L}(H)$ are compatible, 
i.e. the inclusion and orthogonal relations both are contained in the compatibility relation.

Every closed subspace of $H$ can be identified with the (orthogonal) projection on this subspace.
Two closed subspaces are compatible if and only if the corresponding projections commute.
This observation can be generalized as follows. 

An {\it observable} in quantum mechanics  is a measure $\mu$ defined on the $\sigma$-algebra of Borel subsets in ${\mathbb R}$
which takes values in the logic ${\mathcal L}(H)$ and such that $\mu (E),\mu(F)$ are orthogonal for any pair of disjoint Borel subsets $E,F$.
Two observables $\mu$ and $\lambda$ are called {\it compatible} if $\mu (E)$ and $\lambda (F)$ are compatible for any pair of Borel subsets $E,F$.
By the spectral theorem, there is a one-to-one correspondence between observables and self-adjoint operators on $H$.
Two observables are compatible if and only if the corresponding operators commute (von Neumann's theorem).

A set consisting of mutually compatible propositions will be called {\it compatible}.
For every orthogonal basis of $H$ any two closed subspaces spanned by subsets of this basis are compatible.
Every maximal compatible subset of ${\mathcal L}(H)$  consists of all closed subspaces spanned by subsets of a certain orthogonal basis for $H$.
The family of all such subsets coincides with the family of maximal classical logics contained in ${\mathcal L}(H)$.
The logic ${\mathcal L}(H)$ together with the family of maximal classical sublogics is a structure similar 
to the Tits buildings of general linear groups.
A {\it building} is a combinatoric construction defined for any group admitting so-called ${\rm BN}$-pair.
This is an abstract simplicial complex together with a family of distinguished subcomplexes called {\it apartments},
see \cite{Tits}.
The building for the group ${\rm GL}_{n}({\mathbb C})$ is formed by all subspaces of ${\mathbb C}^{n}$
and every apartment is formed by all subspaces spanned by subsets of a certain basis for ${\mathbb C}^{n}$.
Maximal compatible subsets of the logic ${\mathcal L}({\mathbb C}^{n})$ correspond to the apartments defined by orthogonal bases.

The automorphism group of the logic ${\mathcal L}(H)$ is a proper subgroup in
the group of all bijective transformations preserving the compatibility relation in both directions.
Consider, for example, the orthogonal complementation $X\to X^{\perp}$ or 
any transformation which transposes some $X\in {\mathcal L}(H)$ with $X^{\perp}$ and leaves fixed all other elements. 
By L. Moln\'ar and P. \v{S}emrl \cite{MolnarSemrl}, if $f$ is a bijective transformation of ${\mathcal L}(H)$
preserving the compatibility relation in both directions, then there is an automorphism $g$ of the logic ${\mathcal L}(H)$
such that for every $X\in {\mathcal L}(H)$ we have either 
$$f(X)=g(X)\;\mbox{ or }\;f(X)=g(X)^{\perp}.$$
The same holds for bijective transformations of ${\mathcal G}_{\infty}(H)$ preserving the compatibility relation in both directions;
this is a simple modification of L. Plevnik's result  \cite{Plevnik}.
Bijective transformations of ${\mathcal G}_{k}(H)$ preserving the compatibility relation in both directions were described by M. Pankov \cite{Pankov4}. 
If $H$ is infinite-dimensional, then every such transformation can be uniquely extended to an automorphism of the logic ${\mathcal L}(H)$.
The  same statement also is proved for the case when the dimension of $H$ is finite and distinct from $2k$ (except one case of small dimension).
In the case when the dimension of $H$ is equal to $2k\ge 8$, there is a result similar to 
the description of compatibility preserving transformations of ${\mathcal G}_{\infty}(H)$.

It was noted above that the compatibility relation is closely connected to the concept of apartment.  
Let ${\mathcal G}_{k}(V)$ be the Grassmannian formed by $k$-dimensional subspaces of a vector space $V$.
For every basis of this vector space  the associated {\it apartment} of ${\mathcal G}_{k}(V)$ consists of all $k$-dimensional subspaces spanned by subsets of this basis.
In the case when $V$ is finite-dimensional, 
apartments of ${\mathcal G}_{k}(V)$ are the intersections of ${\mathcal G}_{k}(V)$  with apartments of  the building for the group ${\rm GL}(V)$.
By \cite{Pankov-book1}, every apartments preserving bijective transformation of ${\mathcal G}_{k}(V)$ 
is induced by a semilinear automorphism of $V$ or a semilinear isomorphisms of $V$ to the dual vector space $V^{*}$
(the second possibility can be realized only in the case when the dimension of $V$ is equal to $2k$).

Similarly, for every orthogonal basis of $H$ 
the associated {\it orthogonal apartment} of ${\mathcal G}_{k}(H)$ is formed by all $k$-dimensional subspaces spanned by subsets of this basis.
Orthogonal apartments can be characterized as maximal compatible subsets of ${\mathcal G}_{k}(H)$.
Therefore, a bijective transformation of ${\mathcal G}_{k}(H)$ preserves the compatibility relation in both directions if and only if 
it preserves the family of orthogonal apartments in both directions.
The description of such transformations is based on some modifications of the methods applied to apartments preserving transformations in \cite{Pankov-book1}.

The compatibility and orthogonality relations have simple interpretations in terms of operator theory. 
Let us identify every closed subspace with the projection on this subspace.
Since two closed subspaces are compatible if and only if the corresponding projections commute, 
compatibility preserving transformations of the standard quantum logic and the associated Grassmannians 
can be considered as commutativity preserving transformations of the corresponding sets of projections.
Similarly, two closed subspaces are orthogonal if and only if the composition of the associated projections is zero.
Therefore, our results concerning orthogonality and compatibility preserving transformations can be reformulated in terms of 
the discipline known as {\it preserver problems} on operator structures.
This area describes transformations of operator spaces (sometimes, as in our case, operator sets) which preserve various types of relations,
see \cite{Molnar-book}.

\section{Lattices of closed subspaces}
\subsection{Lattices}
Let us start from some general definitions.
Let $(X,\le)$ be a partially ordered set. Let also $Y$ be a subset of $X$.
An element $x\in X$ is an {\it upper bound} of $Y$ if $y\le x$ for every $y\in Y$.
An upper bound $x$ of $Y$ is said to be its {\it least upper bound} if $x\le x'$ for every upper bound $x'$ of $Y$.
Dually, $x\in X$ is a {\it lower bound} of $Y$ if $x\le y$ for every $y\in Y$.
We say that a lower bound $x$ of $Y$ is the {\it greatest lower bound} of $Y$ if $x'\le x$ for every lower bound $x'$ of $Y$. 

The partially ordered set $(X,\le)$ is called a {\it lattice} if for any two elements $x,y\in X$
the subset $\{x,y\}$ has the least upper bound denoted by $x\vee y$ and the greatest lower bound denoted by $x\wedge y$.
This lattice is {\it bounded} if it contain the {\it least} element $0$ and the {\it greatest} element $1$
satisfying $0\le x \le 1$ for every $x\in X$.

An {\it isomorphism} between partially ordered sets $(X,\le)$ and $(X',\le)$ is a bijective mapping $f:X\to X'$ preserving 
the order in both  directions, i.e.
$$x\le y\;\Longleftrightarrow\;f(x)\le f(y)$$
for all $x,y\in X$.
If our partially ordered sets are lattices and $f:X\to X'$ is an isomorphism between them,  then
$$f(x\vee y)=f(x)\vee f(y)\;\mbox{ and }\;f(x\wedge y)=f(x)\wedge f(x)$$
for all $x,y\in X$.
Isomorphisms of bounded lattices  preserve the least and greatest elements.

\subsection{Lattices of subspaces of vector spaces}
Let $V$ be a vector space over a field.
Denote by ${\mathcal L}(V)$ the set of all subspaces of $V$ which is partially ordered by the inclusion relation $\subseteq$.
This is a bounded lattice. 
For any two subspaces $X,Y$ the least upper bound $X\vee Y$ and the greatest lower bound $X\wedge Y$
coincide with the sum $X+Y$ and the intersection $X\cap Y$, respectively. 
The least element is $0$ and the greatest element is $V$.
In the case when $\dim V=1$, the lattice is trivial, i.e. it consists of the least element and the greatest element only.
For this reason, we will always suppose that $\dim V\ge 2$.
For every natural $k<\dim V$
the lattice ${\mathcal L}(V)$ contains the following two subsets:
\begin{enumerate}
\item[$\bullet$] the Grassmannian ${\mathcal G}_{k}(V)$ formed by $k$-dimensional subspaces,
\item[$\bullet$] the Grassmannian ${\mathcal G}^{k}(V)$ formed by subspaces of codimension $k$.
\end{enumerate}
If $\dim V=n$ is finite, then ${\mathcal G}^{k}(V)$ coincides with ${\mathcal G}_{n-k}(V)$.
In the case when $V$ is infinite-dimensional, there is also the subset ${\mathcal G}_{\infty}(V)$
consisting of all subspaces whose dimension and codimension both are infinite.
This subset is homogeneous, i.e. for any two elements of ${\mathcal G}_{\infty}(V)$
there is a linear automorphism of $V$ transferring one of them to the other, 
only in the case when the dimension of $V$ is the smallest infinite cardinal number $\aleph_{0}$.

Let $V$ and $V'$ be vector spaces over fields ${\mathbb F}$ and ${\mathbb F}'$, respectively. 
We say that a mapping $L:V\to V'$ is {\it semilinear} if 
$$L(x+y)=L(x)+L(y)$$
for all vectors $x,y\in V$ and there is an isomorphism $\sigma:{\mathbb F}\to {\mathbb F}'$ such that 
$$L(ax)=\sigma(a)L(x)$$
for every vector $x\in V$ and every scalar $a\in {\mathbb F}$.
This mapping is linear if the fields are coincident  and $\sigma$ is  identity. In the general case, it is said to be $\sigma$-{\it linear}.
There are semilinear mappings associated to non-surjective field homomorphisms \cite{FaureFrolicher-book,Pankov-book2},
but we do not consider them here.
Semilinear bijections are called {\it semilinear isomorphisms}.
Every semilinear isomorphism $L:V\to V'$ induces an isomorphism of the lattice ${\mathcal L}(V)$ to the lattice ${\mathcal L}(V')$.
Every non-zero scalar multiple of $L$ is a semilinear isomorphism 
which induces the same lattice isomorphism.
Conversely, if two semilinear isomorphisms define the same isomorphism between the lattices,
then one of them is a scalar multiple of the other.

Suppose that $L:V\to V'$ and $L':V\to V'$ are semilinear isomorphisms which induce 
the same bijection of ${\mathcal G}_{1}(V)$ to ${\mathcal G}_{1}(V')$, i.e. we have $L(P)=L'(P)$ for every $P\in {\mathcal G}_{1}(V)$.
Then for every non-zero vector $x\in V$ there is a scalar $a_{x}$ such that $$L'(x)=a_{x}L(x).$$
If $x,y\in V$ are linearly independent, then 
$$a_{x}L(x)+a_{y}L(y)=L'(x+y)=a_{x+y}(L(x)+L(y))$$
and $a_{x}=a_{x+y}=a_{y}$, since $L(x)$ and $L(y)$ are linearly independent.
If $y$ is a scalar multiple of $x$, then we take any vector $z\in V$ such that $x,z$ are linearly independent 
(this is possible, since $\dim V\ge 2$)
and establish that $a_{x}=a_{z}=a_{y}$. So, we have $a_{x}=a_{y}$ for any two non-zero vectors $x,y\in V$
which means that $L'$ is a scalar multiple of $L$.

In the case when $\dim V\ge 3$, every isomorphism between the lattices ${\mathcal L}(V)$ and ${\mathcal L}(V')$ is induced by a semilinear isomorphism.
This is a simple consequence of the Fundamental Theorem of Projective Geometry which will be given below.

For every $2$-dimensional subspace $S\subset V$
the set formed by all $1$-dimensional subspaces contained in $S$, i.e. ${\mathcal G}_{1}(S)$, is called a {\it line} of ${\mathcal G}_{1}(V)$.
In the case when $\dim V\ge 3$, the Grassmannian ${\mathcal G}_{1}(V)$ together with all such lines is known as the {\it projective space} associated to $V$.
We denote this projective space by $\Pi_{V}$.
For $\dim V=2$ there is only one line and we exclude this case.

An {\it isomorphism} of the projective space $\Pi_{V}$ to the projective space $\Pi_{V'}$
is a bijection $f:{\mathcal G}_{1}(V)\to{\mathcal G}_{1}(V')$ such that $f$ and $f^{-1}$ send lines to lines.
 
\begin{theorem}[The Fundamental Theorem of Projective Geometry]\label{theorem-FTPG}
Suppose that the dimensions of $V$ and $V'$ both are not less than $3$.
Then every isomorphism of $\Pi_{V}$ to $\Pi_{V'}$ is induced by a semilinear isomorphism $L:V\to V'$
and any other semilinear isomorphism inducing this isomorphism of projective spaces is a scalar multiple of $L$.
\end{theorem}

\begin{proof}
See, for example, \cite[Section III.1, p.44]{Baer}.
We also refer \cite{FaureFrolicher-book, Pankov-book1} or the original research articles \cite{FaureFrolicher,Havlicek}
for a more general version of this result.
\end{proof}

\begin{rem}\label{rem2-1}{\rm
If $V$ and $V'$ are of the same finite dimension not less than $3$, then every bijection of ${\mathcal G}_{1}(V)$ to ${\mathcal G}_{1}(V')$
sending lines to subsets of lines is an isomorphism of $\Pi_{V}$ to $\Pi_{V'}$ \cite[Theorem 2.26]{Artin}.
}\end{rem}

\begin{cor}\label{cor-FTPG}
If $\dim V\ge 3$, then every isomorphism $f$ of the lattice ${\mathcal L}(V)$ to the lattice ${\mathcal L}(V')$ 
is induced by a semilinear isomorphism $L:V\to V'$
and any other semilinear isomorphism inducing $f$ is a scalar multiple of $L$.
\end{cor}

\begin{proof}
Since $f$ preserves the inclusion relation in both directions, 
it transfers every ${\mathcal G}_{k}(V)$ to ${\mathcal G}_{k}(V')$.
Therefore, $\dim V'\ge 3$ and the restriction of $f$ to ${\mathcal G}_{1}(V)$ is an isomorphism of $\Pi_{V}$ to $\Pi_{V'}$.
By Theorem \ref{theorem-FTPG}, there exists a semilinear isomorphism $L:V\to V'$ such that 
$$f(X)=L(X)\;\mbox{ for every }\;X\in {\mathcal G}_{1}(V).$$
For any subspace $X\subset V$ we have 
$${\mathcal G}_{1}(f(X))=f({\mathcal G}_{1}(X))=L({\mathcal G}_{1}(X))={\mathcal G}_{1}(L(X))$$
which implies that $f(X)$ coincides with $L(X)$.
If $f$ is also induced by a semilinear isomorphism $L':V\to V'$, then $L(X)=L'(X)$ for every $X\in {\mathcal G}_{1}(V)$
and Theorem \ref{theorem-FTPG} implies that $L'$ is a scalar multiple of $L$.
\end{proof}

In the case when $\dim V=2$, every bijective transformation of ${\mathcal L}(V)$ preserving $0$ and $V$ is an automorphism of 
the lattice and the above statement fails.

\subsection{Linear and conjugate-linear operators}
The automorphism group of the field of real numbers is trivial and all semilinear mappings between real vector spaces are linear.
The automorphism group of the field of complex numbers contains the conjugation $a\to \overline{a}$
and infinitely many other automorphisms.

\begin{exmp}{\rm
Using Zorn's lemma and \cite[Chapter V, Theorem 2.8]{Lang}, we can show that
every automorphism of a field can be extended to an automorphism of any algebraically closed extension of this field
(see, for example, \cite[Section 1.1]{Pankov-book2}).
The field ${\mathbb Q}(\sqrt{p})$ ($p$ is a prime number) is contained in the algebraically closed field ${\mathbb C}$.
Consider the automorphism of ${\mathbb Q}(\sqrt{p})$ sending every $v+w\sqrt{p}$ to $v-w\sqrt{p}$
and extend it to an automorphism of ${\mathbb C}$.
Any such extension is not identity on ${\mathbb R}$ which implies that it is different from the conjugation.
}\end{exmp}

\begin{lemma}\label{lemma2-1}
Every continuous automorphism of the field ${\mathbb C}$ is identity or the conjugation.
\end{lemma}

\begin{proof}
If $\sigma$ is an automorphism of ${\mathbb C}$, then the restriction of $\sigma$ to ${\mathbb Q}$ is identity.
In the case when $\sigma$ is continuous, its restriction to ${\mathbb R}$ is identity.
It is clear that $\sigma({\mathbf i})=\pm {\mathbf i}$ and we get the claim.
\end{proof}

Let $H$ and $H'$ be complex Hilbert spaces.
A semilinear mapping $L:H\to H'$ is {\it bounded} if there is a nonnegative real number $a$ such that 
$$||L(x)||\le a||x||$$
for all vectors $x\in H$.
The smallest number $a$ satisfying this condition is called the {\it norm} of $L$ and denoted by $||L||$.

\begin{prop}\label{prop2-1}
For every bounded semilinear mapping of $H$ to $H'$ the associated automorphism of the field ${\mathbb C}$ is identity or the conjugation.
\end{prop}

This is a simple consequence of the following.

\begin{lemma}\label{lemma2-2}
If $\sigma$ is an automorphism of the field ${\mathbb C}$ such that for every sequence of complex numbers $\{a_{n}\}_{n\in {\mathbb N}}$
converging to $0$ the sequence $\{\sigma(a_{n})\}_{n\in {\mathbb N}}$ is bounded, then $\sigma$ is identity or the conjugation.
\end{lemma}

\begin{proof}
By Lemma \ref{lemma2-1}, we need to show that $\sigma$ is continuous.
Since $\sigma$ is additive, it is sufficient to establish that $\sigma$ is continuous in $0$. 
Indeed, if $a_{n}\to a$, then $(a_{n}-a)\to 0$ and $\sigma(a_{n}-a)\to 0$ implies that $\sigma(a_{n})\to a$.

If a sequence $\{a_{n}\}_{n\in {\mathbb N}}$ converges  to $0$ and $\{\sigma(a_{n})\}_{n\in {\mathbb N}}$ is not converging to $0$,
then $\{a_{n}\}_{n\in {\mathbb N}}$ contains a subsequence $\{a'_{n}\}_{n\in {\mathbb N}}$ such that 
the inequality $|\sigma(a'_{n})|>a$ holds for a certain real number $a>0$ and all natural $n$.
In the sequence $\{a'_{n}\}_{n\in {\mathbb N}}$, we  choose a subsequence $\{a''_{n}\}_{n\in {\mathbb N}}$ satisfying $na''_{n}\to 0$.
Recall that $\sigma$ is an automorphism of ${\mathbb C}$ and we have $\sigma(n)=n$ for every natural $n$. 
Then $$|\sigma(na''_{n})|=n|\sigma(a''_{n})|>na$$ and 
the sequence $\{\sigma(na''_{n})\}_{n\in {\mathbb N}}$ is unbounded which contradicts our assumption.
\end{proof}

Linear mappings of $H$ to $H'$ will be called {\it linear operators}.
A semilinear mapping of $H$ to $H'$ is said to be a {\it conjugate-linear operator} if 
the associated automorphism of ${\mathbb C}$ is the conjugation.
If a linear or conjugate-linear operator $A:H\to H'$ is invertible, 
then the operators $A$ and $A^{-1}:H'\to H$ are of the same type, i.e. both are linear or conjugate-linear.

\begin{exmp}{\rm
It is well-known that every linear operator on the Hilbert space ${\mathbb C}^n$ is bounded. 
The mapping 
$$x=(x_{1},\dots,x_{n})\to \overline{x}=(\overline{x}_1,\dots,\overline{x}_n)$$
is an invertible bounded conjugate-linear operator on ${\mathbb C}^{n}$.
Every conjugate-linear operator on ${\mathbb C}^{n}$ is of type $x\to \overline{A(x)}$, where $A$ is a linear operator.
Therefore, each conjugate-linear operator on ${\mathbb C}^n$ is bounded.
}\end{exmp}

\begin{exmp}\label{exmp-conjlin}{\rm
Let $B=\{e_{i}\}_{i\in I}$ be an orthonormal basis of $H$. 
There is the unique conjugate-linear operator $C_{B}$ which leaves fixed every vector from this basis.
If $J$ is a countable or finite subset of $I$ and 
$x=\sum_{j\in J}a_{j}e_{j}$,
then 
$$C_{B}(x)=\sum_{j\in J}\overline{a}_{j}e_{j}.$$
Every conjugate-linear operator $A':H'\to H$ can be presented as the composition $C_{B}A$, where $A:H'\to H$ is a linear operator.
The operator $C_B$ is invertible bounded.
}\end{exmp}

We will exploit the following well-known operator properties:
\begin{enumerate}
\item[$\bullet$]
A linear operator is bounded if and only if transfers bounded subsets to bounded subsets.
\item[$\bullet$]
By the bounded inverse theorem, for every invertible bounded linear operator $A$ the inverse operator $A^{-1}$ is bounded.
\end{enumerate}
Using the operator $C_{B}$ from Example \ref{exmp-conjlin}, we can show that the same statements hold for conjugate-linear operators.

If $A:H\to H'$ is a bounded linear operator, then
for  every vector $y\in H'$ the mapping $$x\to \langle A(x), y\rangle$$ is a bounded linear functional on $H$
and, by Riesz's representation theorem, there exists the unique vector $A^{*}(y)\in H$ such that 
$$\langle A(x),y\rangle=\langle x, A^{*}(y)\rangle$$
for all vectors $x\in H$.
The mapping $A^{*}:H'\to H$ is a bounded linear operator and $||A^{*}||=||A||$.
This operator is known as {\it adjoint} to $A$. 

Now, we suppose that $A:H\to H'$ is a bounded conjugate-linear operator.
For every vector $y\in H'$ the mapping $$x\to \overline{\langle A(x),y\rangle}$$ is a bounded linear functional on $H$ 
and there is the unique vector $A^{*}(y)\in H$ such that 
$$\overline{\langle A(x),y\rangle}=\langle x,A^{*}(y)\rangle$$
for all vectors $x\in H$. 
We get a conjugate-linear operator $A^{*}:H'\to H$ which will be called {\it adjoint} to $A$. 
As above, the operator $A^{*}$ is bounded and $||A^{*}||=||A||$.

For every linear or conjugate-linear bounded operator $A:H\to H'$ we have $A^{**}=A$.
Also, $A^{*}$ is invertible if and only if $A$ is invertible.
In this case, the operators $(A^{-1})^{*}$ and $(A^{*})^{-1}$ are coincident.

An invertible linear operator $U$ on $H$  is {\it unitary} if for all vectors $x,y\in H$ we have 
$$\langle U(x),U(y)\rangle=\langle x,y\rangle.$$
An invertible conjugate-linear operator $U$ on $H$ is said to be {\it anti-unitary} if 
$$\langle U(x),U(y)\rangle=\overline{\langle x,y\rangle}$$
for all vectors $x,y\in H$.
In each of these cases, $||U(x)||=||x||$ for every vector $x\in H$.
In particular, unitary and anti-unitary operators are bounded and transfer orthonormal bases to orthonormal bases.
An invertible bounded linear operator $U$ is unitary if and only if  $U^{-1}=U^{*}$.
Similarly, an invertible bounded conjugate-linear operator $U$ is anti-unitary if and only if the same equality holds.

\begin{exmp}{\rm
The normalized Fourier transform on the Hilbert space $L^{2}({\mathbb R})$ is an unitary operator.
The operator $C_{B}$ from Example \ref{exmp-conjlin} is anti-unitary
and every anti-unitary operator on $H$ can be presented as the composition $C_{B}U$, where $U$ is an unitary operator on $H$.
}\end{exmp}

\begin{exmp}\label{exmp-proj}{\rm
A linear operator $P$ on $H$ is called an {\it idempotent} if $P^{2}=P$.
The latter equality implies that the restriction of $P$ to the image ${\rm Im}(P)$ is identity and
for every $x\in H$ the vector $x-P(x)$ belongs to the kernel ${\rm Ker}(P)$.
This means that $H$ is the direct sum of the subspaces ${\rm Ker}(P)$ and ${\rm Im}(P)$, i.e.
every vector $x\in H$ can be uniquely presented as the sum of $y\in {\rm Ker}(P)$ and $z\in {\rm Im}(P)$
such that $P(x)=z$.
The operator ${\rm Id}_{H}-P$ also is an idempotent and
$${\rm Ker}({\rm Id}_{H}-P)={\rm Im}(P),\;\;{\rm Im}({\rm Id}_{H}-P)={\rm Ker}(P).$$
Conversely, if $H$ is the direct sum of subspaces $S$ and $U$,
then for every vector $x\in H$ there are the unique $x_{S}\in S$ and $x_{U}\in U$ such that  $x=x_{S}+x_{U}$ and 
the operators $x\to x_{S}$ and $x\to x_{U}$ are idempotents.
If an idempotent $P$ is bounded, then the subspaces ${\rm Ker}(P)$ and ${\rm Im}(P)$ are closed
and the adjoint operator $P^{*}$ also is an idempotent.
Bounded self-adjoint idempotents are called {\it projections}. 
An idempotent $P$ is a projection if and only if the subspaces ${\rm Ker}(P)$ and ${\rm Im}(P)$ are orthogonal.
The Hilbert space $L^{2}({\mathbb R})$ is the orthogonal direct sum of the subspace of even functions and the subspace of odd functions;
the corresponding projections are 
$$f\to\frac{f(x)+f(-x)}{2}\;\mbox{ and }\;f\to\frac{f(x)-f(-x)}{2}.$$
}\end{exmp}

\begin{exmp}\label{exmp-inv}{\rm
A non-identity operator $S$  is called an {\it involution} if $S^{2}={\rm Id}_{H}$.
For every idempotent $P$ the operator $S={\rm Id}_{H}-2P$ is an involution;
the restriction of $S$ to ${\rm Ker}(P)$ is identity and $S(x)=-x$ for all vectors $x\in {\rm Im}(P)$.
Conversely, for every involution $S$ the operator $\frac{1}{2}({\rm Id}_{H}-S)$ is an idempotent.
So, there is a one-to-one correspondence between idempotents and involutions.
An involution is unitary if and only if the associated idempotent is a projection. 
Every unitary involution is self-adjoint.
}\end{exmp}

\subsection{Lattices of closed subspaces of Hilbert spaces}
Let $H$ be a complex Hilbert space.
Denote by ${\mathcal L}(H)$ the set of all closed subspaces of $H$ which is partially ordered by the inclusion relation $\subseteq$.
This is a bounded lattice whose least element is $0$ and whose greatest element is $H$.
For any two closed subspaces $X,Y$ the greatest lower bound is the intersection $X\cap Y$ 
and the least upper bound is $X\dotplus Y$, i.e. the minimal closed subspace containing $X+Y$.

\begin{exmp}{\rm
The sum $X+Y$ of closed subspaces $X$ and $Y$ is not necessarily closed.
Let $\{e_{n},f_{n}\}_{n\in \mathbb N}$ be an orthonormal basis for $H$.
Suppose that $X$ is the closed subspace spanned by all $e_{n}$ and 
$Y$ is the closed subspace spanned by all $f_{n}+ne_{n}$.
Every vector from the basis belongs to $X+Y$ which means that $X\dotplus Y$ coincides with $H$.
However, $X+Y$ is a proper subspace of $H$.
Indeed, the vector defined by a series $\sum^{\infty}_{n=1}a_{n}f_{n}$
belongs to $X+Y$ if and only if  the series 
$\sum^{\infty}_{n=1}a_{n}ne_{n}$
is convergent (we leave all details for the readers).
On the other hand, we can state that $X+Y$ is closed if one of the following possibilities is realized:
\begin{enumerate}
\item[$\bullet$] at least one of the subspaces is finite-dimensional,
\item[$\bullet$] $(X\cap Y)^{\perp}\cap X$ and $(X\cap Y)^{\perp}\cap Y$ are orthogonal, in particular,
if $X,Y$ are orthogonal.
\end{enumerate}
}\end{exmp}

As above, we write ${\mathcal G}_{k}(H)$ for the Grassmannian formed by $k$-dimensional subspaces of $H$.
Since all finite-dimensional subspaces of $H$ are closed,  every such Grassmannian  is contained in ${\mathcal L}(H)$.
Let ${\mathcal G}^{k}(H)$ be the Grassmannian consisting of closed subspaces whose codimension is equal to $k$.
In the case when $\dim V=n$ is finite, ${\mathcal G}^{k}(H)$ coincides with ${\mathcal G}_{n-k}(H)$.

If $H$ is infinite-dimensional, then we write ${\mathcal G}_{\infty}(H)$ for the set of all closed subspaces 
whose dimension and codimension both are infinite.
This subset is homogeneous, i.e. for any two elements of ${\mathcal G}_{\infty}(H)$
there is an invertible bounded linear operator on $H$ transferring one of them to the other, 
only in the case when $H$ is separable.

Let $H$ and $H'$ be complex Hilbert spaces.
Every invertible bounded linear or conjugate-linear operator $A:H\to H'$ induces an isomorphism of the lattice 
${\mathcal L}(H)$ to the lattice ${\mathcal L}(H')$ and
any non-zero scalar multiple of $A$ gives the same lattice isomorphism.

The mapping $X\to X^{\perp}$ is a bijective transformation of ${\mathcal L}(H)$ reversing the inclusion relation.
It sends every ${\mathcal G}_{k}(H)$ to ${\mathcal G}^{k}(H)$ and conversely.

\begin{prop}\label{prop2-2}
For every invertible bounded linear or conjugate-linear operator $A:H\to H'$ the mapping of ${\mathcal L}(H)$ to ${\mathcal L}(H')$ defined as
$$X\to A(X^{\perp})^{\perp}$$
is the lattice isomorphism induced by $(A^{*})^{-1}=(A^{-1})^{*}$.
\end{prop}
\begin{proof}
If $x,y\in H$, then $\langle y,x \rangle=\langle A^{-1}A(y), x\rangle$ is equal to 
$$\langle A(y), (A^{-1})^{*}(x)\rangle\;\mbox{ or }\;\overline{\langle A(y), (A^{-1})^{*}(x)\rangle}.$$
In other word, $y$ is orthogonal to $x$ if and only if $A(y)$ is orthogonal to $(A^{-1})^{*}(x)$.
This implies that $A(X^{\perp})^{\perp}$ coincides with $(A^{-1})^{*}(X)$ for every closed subspace $X\subset H$.
\end{proof}

The main result of this section is the following.

\begin{theorem}[G. W. Mackey \cite{Mackey}, S. Kakutani and G. W. Mackey \cite{KakutaniMackey}]\label{theorem-mackey}
Suppose that $H$ and $H'$ are infinite-dimensional. 
Then every isomorphism of the lattice ${\mathcal L}(H)$ to the lattice ${\mathcal L}(H')$ is induced by 
an invertible bounded linear or conjugate-linear operator $A: H\to H'$
and any other operator inducing this lattice isomorphism is a scalar multiple of $A$.
\end{theorem}

For the finite-dimensional case this statement fails.
If $H$ is finite-dimensional, then ${\mathcal L}(H)$ consists of all subspaces of $H$
and there are the automorphisms of ${\mathcal L}(H)$ induced by unbounded semilinear automorphisms of $H$, i.e.
semilinear automorphisms associated to non-continuous automorphisms of the field ${\mathbb C}$.

If $H$ is infinite-dimensional, then ${\mathcal G}_{\infty}(H)$ is partially ordered by the inclusion relation,
but it is not a lattice (there exist pairs $X,Y\in {\mathcal G}_{\infty}(H)$ such that 
$X\cap Y$ is finite-dimensional or $X\dotplus Y$ is of finite codimension).
The next result concerns isomorphisms between such partially ordered sets.

\begin{theorem}[M. Pankov \cite{Pankov2,Pankov3}]\label{theorem-pank1}
Suppose that $H$ and $H'$ are infinite-dimensional.
Then every isomorphism of the partially ordered set ${\mathcal G}_{\infty}(H)$ to the partially ordered set ${\mathcal G}_{\infty}(H')$
can be uniquely extended to an isomorphism of the lattice ${\mathcal L}(H)$ to the lattice ${\mathcal L}(H')$.
\end{theorem}

\begin{rem}\label{rem-projections}{\rm
As above, we suppose that $H$ is infinite-dimensional.
Let ${\mathcal I}(H)$ be the set consisting of all bounded idempotents on $H$.
This is a partially ordered set: for $P,Q\in {\mathcal I}(H)$ we have $P\le Q$ if 
$${\rm Im}(P)\subset {\rm Im}(Q)\;\mbox{ and }\;{\rm Ker}(Q)\subset {\rm Ker}(P).$$
Since every closed subspace $X\subset H$ can be identified with the projection whose image is $X$,
the lattice ${\mathcal L}(H)$ is contained in the partially ordered set ${\mathcal I}(H)$.
By P. G. Ovchinikov \cite{Ovch}, every automorphism of the partially ordered set ${\mathcal I}(H)$ is of type 
$$P\to APA^{-1}\;\mbox{ or }\;P\to AP^{*}A^{-1},$$
where $A$ is an invertible bounded linear or conjugate-linear operator on $H$.
L. Plevnik \cite{Plevnik} considered the partially ordered set ${\mathcal I}_{\infty}(H)$ 
formed by all idempotents from ${\mathcal I}(H)$ whose image and kernel both are infinite-dimensional.
One of his results states that every automorphism of this partially ordered set 
can be uniquely extended to an automorphism of the partially ordered set ${\mathcal I}(H)$. 
}\end{rem}

\subsection{Proof of Theorem \ref{theorem-mackey}}
In this and the next sections, we will suppose that $H$ and $H'$ are infinite-dimensional complex Hilbert spaces.

\begin{lemma}\label{lemma2-3}
There is a sequence of vectors $\{x_{n}\}_{n\in {\mathbb N}}$ in $H$ satisfying the following condition:
for every bounded sequence of complex numbers $\{a_{n}\}_{n\in {\mathbb N}}$
there exists a vector $x\in H$
such that $\langle x_{n},x\rangle =a_{n}$ for every $n$.
\end{lemma}

\begin{proof}
Let $\{e_{n}\}_{n\in {\mathbb N}}$ be a sequence formed by unit mutually orthogonal vectors of $H$.
We set $x_{n}=2^{n}e_{n}$.
Then for every bounded sequence of complex numbers $\{a_{n}\}_{n\in {\mathbb N}}$ the vector
$$x=\sum^{\infty}_{n=1}\frac{\overline{a}_{n}}{2^{n}}e_{n}$$
is as required.
\end{proof}

\begin{lemma}\label{lemma2-4}
If a semilinear isomorphism $L:H\to H'$ sends every closed subspace of codimension $1$ to a closed subspace,
then $L$ is linear or conjugate-linear.
\end{lemma}

\begin{proof}
Let $\sigma$ be the automorphism of ${\mathbb C}$ associated to $L$.
By Lemma \ref{lemma2-2}, we need to show that for every sequence of complex number $\{a_{n}\}_{n\in {\mathbb N}}$ 
converging to $0$ the sequence $\{\sigma(a_{n})\}_{n\in {\mathbb N}}$ is bounded.
If the latter sequence is unbounded,
then $\{a_{n}\}_{n\in {\mathbb N}}$ contains a subsequence $\{b_{n}\}_{n\in {\mathbb N}}$
such that
\begin{equation}\label{eq2-1}
|\sigma(b_{n})|\ge n||L(x_{n})||\;\mbox{ for every }\;n\in {\mathbb N},
\end{equation}
where $\{x_{n}\}_{n\in {\mathbb N}}$ is the sequence from Lemma \ref{lemma2-3}.
Consider a vector $x\in H$ satisfying $\langle x_{n},x\rangle =b_{n}$ for every $n$ and take $x'=x/||x||^2$.
Then 
$$x_{n}=y_{n}+b_{n}x',$$ 
where each $y_{n}$ is a vector orthogonal to $x$, and
$$L(x_{n})/\sigma(b_{n})=L(y_{n}/b_{n})+L(x').$$
It follows from \eqref{eq2-1} that $L(x_{n})/\sigma(b_{n})\to 0$. Then the latter equality implies that
$$L(-y_{n}/b_{n})\to L(x').$$
Therefore, $L(x')$ belongs to the closure of $L(x^{\perp})$
(recall that every $y_{n}$ is contained in $x^{\perp}$).
By our hypothesis, the subspace $L(x^{\perp})$ is closed.
Hence $L(x')$ belongs to $L(x^{\perp})$ and we have $x'\in x^{\perp}$ which is impossible,
since $x'$ is a non-zero scalar multiple of $x$.
This contradiction gives the claim. 
\end{proof}

Let $f$ be an isomorphism of the lattice ${\mathcal L}(H)$ to the lattice ${\mathcal L}(H')$.
Then $f$ transfers ${\mathcal G}_{k}(H)$ to ${\mathcal G}_{k}(H')$ and ${\mathcal G}^{k}(H)$ to ${\mathcal G}^{k}(H')$.
Therefore, the restriction of $f$ to ${\mathcal G}_{1}(H)$ is an isomorphism between the projective spaces $\Pi_H$ and $\Pi_{H'}$.
Theorem \ref{theorem-FTPG} implies the existence of a semilinear isomorphism $A:H\to H'$ such that 
$$f(X)=A(X)\;\mbox{ for every }\;X\in {\mathcal G}_{1}(H).$$
As in the proof of Corollary \ref{cor-FTPG}, we establish that the same equality holds for every closed subspace $X\subset H$
and any other semilinear isomorphism inducing $f$ is a scalar multiple of $A$.

Since $f$ sends ${\mathcal G}^{1}(H)$ to ${\mathcal G}^{1}(H')$,
Lemma \ref{lemma2-4} guarantees that $A$ is linear or conjugate-linear.
Now, we prove that it transfers bounded subsets to bounded subsets which implies that $A$ is bounded.
Let $X$ be a bounded subset of $H$. It is sufficient to show that the subset 
\begin{equation}\label{eq2-2}
\{\langle A(x),x'\rangle :x\in X\}
\end{equation}
is bounded in ${\mathbb C}$ for every vector $x'\in H'$. 
This guarantees that $A(X)$ is bounded (every weakly bounded subset is bounded).

For every vector $x'\in H'$ there exists a vector $y\in H$ such that 
$$A(y^{\perp})=(x')^{\perp}.$$
Let us fix a vector $z\in H$ satisfying $\langle z,y\rangle =1$.
Any vector $x\in H$ can be presented in the form $$x=x_{0}+ \langle x,y\rangle z,$$
where $x_{0}$ is a vector orthogonal to $y$.
Then
$$\langle A(x),x'\rangle =\langle A(x_{0}),x'\rangle + \langle A(\langle x,y\rangle z),x'\rangle.$$
Since $x_{0}\in y^{\perp}$, we have $A(x_{0})\in A(y^{\perp})=(x')^{\perp}$.
This means that 
$$\langle A(x),x'\rangle =a \langle A(z),x'\rangle ,$$ 
where $a$ is equal to $\langle x,y \rangle$ or $\overline{\langle x,y\rangle}$.
So, 
$$|\langle A(x),x'\rangle |=|\langle x,y\rangle |\cdot|\langle A(z),x'\rangle|.$$
The latter implies that the subset \eqref{eq2-2} is bounded, since $\{\langle x,y\rangle : x\in X\}$ is bounded.

\begin{rem}\label{rem-mackey}{\rm
Theorem \ref{theorem-mackey} was first proved by G. W. Mackey in \cite{Mackey} for the lattices of closed subspaces of infinite-dimensional real normed spaces.
Lemma \ref{lemma2-4} was obtained in \cite{KakutaniMackey}.
It shows that the arguments given in \cite{Mackey} work for the complex case.
See also \cite{FillmoreLongstaff}.
}\end{rem}

\subsection{Proof of Theorem \ref{theorem-pank1}}
Let $f$ be an isomorphism of the partially ordered set ${\mathcal G}_{\infty}(H)$ to the partially ordered set ${\mathcal G}_{\infty}(H')$.

\begin{lemma}\label{lemma2-5}
For every $X\in {\mathcal G}_{\infty}(H)$ there is an invertible bounded linear or conjugate-linear operator $A_{X}:X\to f(X)$ such that
$$f(Y)=A_{X}(Y)$$ for every $Y\in {\mathcal G}_{\infty}(H)$ contained in $X$.
\end{lemma}

\begin{proof}
Let ${\mathcal X}$ be the set of all elements of ${\mathcal G}_{\infty}(H)$ contained in $X$.
Then $f({\mathcal X})$ consists of all elements of ${\mathcal G}_{\infty}(H)$ contained in $X'=f(X)$.
We consider the closed subspaces $X$ and $X'$ as Hilbert spaces and
write $Y^{\perp}$ and $Y'^{\perp}$ for the orthogonal complements of $Y\subset X$ and $Y'\subset X'$ in these Hilbert spaces.
Denote by ${\mathcal Y}$ and ${\mathcal Y}'$ the sets formed by closed subspaces of infinite codimension in $X$ and $X'$, respectively.
Then
$$Y\in{\mathcal Y}\;\Longleftrightarrow Y^{\perp}\in {\mathcal X}\;\mbox{ and }\;
Y'\in {\mathcal Y}'\;\Longleftrightarrow\;Y'^{\perp}\in f({\mathcal X}).$$
The bijection $g:{\mathcal Y}\to {\mathcal Y}'$ sending
every $Y$ to $f(Y^{\perp})^{\perp}$ is order preserving in both directions.
It is clear that 
$$g({\mathcal G}_{k}(X))={\mathcal G}_{k}(X')$$
for every natural $k$.
In particular, the restriction of $g$ to ${\mathcal G}_{1}(X)$ is an isomorphism between the projective spaces $\Pi_X$ and $\Pi_{X'}$.
So, there is a semilinear isomorphism $L:X\to X'$ such that 
$$g(Y)=L(Y)\;\mbox{ for every }\;Y\in {\mathcal G}_{1}(X).$$ 
As above, we establish that this equality holds for all $Y\in {\mathcal Y}$.
Also, for every $Y\in {\mathcal Y}$ the lattice ${\mathcal L}(Y)$ is contained in ${\mathcal Y}$
and the restriction of $g$ to this lattice is an isomorphism to the lattice ${\mathcal L}(g(Y))$.
Theorem \ref{theorem-mackey} implies that $L$ is bounded on any infinite-dimensional subspace $Y\in {\mathcal Y}$.
Since $X$ can be presented as the orthogonal sum of two elements from ${\mathcal Y}$, the similinear mapping $L:X\to X'$ is bounded. 
So, $L:X\to X'$ is an invertible bounded linear or conjugate-linear operator such that 
$$f(Y^{\perp})^{\perp}=L(Y)$$
for every $Y\in {\mathcal Y}$.
Proposition \ref{prop2-2} shows that the operator $A_{X}=(L^{*})^{-1}$ satisfies the required condition.
\end{proof}

\begin{lemma}\label{lemma2-6}
Let $X$ and $Y$ be elements of ${\mathcal G}_{\infty}(H)$ satisfying
$$\dim (X\cap Y)<\infty.$$
Then there exists $Z\in {\mathcal G}_{\infty}(H)$ such that 
$X\cap Z$ and $Y\cap Z$ are elements of ${\mathcal G}_{\infty}(H)$ containing $X\cap Y$.
\end{lemma}

\begin{proof}
Let $X'$ and $Y'$ be the orthogonal complements of $X\cap Y$ in $X$ and $Y$, respectively.
We choose inductively a sequence of mutually orthogonal vectors $\{x_{n},x'_{n},y_{n},y'_{n}\}_{n\in {\mathbb N}}$ such that 
$$x_{n},x'_{n}\in X'\;\mbox{ and }\;y_{n},y'_{n}\in Y'$$ 
for every $n$.
Denote by $Z'$ the closed subspace spanned by $\{x'_{n},y'_{n}\}_{n\in {\mathbb N}}$.
The subspace $Z=(X\cap Y)+Z'$ is as required.
\end{proof}

Let $P$ be a $1$-dimensional subspace of $H$. 
We take any $X\in {\mathcal G}_{\infty}(H)$ containing $P$ and set
$$g(P)=A_{X}(P).$$
We need to show that the definition of $g(P)$ does not depend on the choice of $X\in{\mathcal G}_{\infty}(H)$.

Suppose that $P$ is contained in $X\in {\mathcal G}_{\infty}(H)$ and $Y\in {\mathcal G}_{\infty}(H)$.
In the case when $X\cap Y$ is an element of ${\mathcal G}_{\infty}(H)$, 
we choose $X',Y'\in {\mathcal G}_{\infty}(H)$ contained in $X\cap Y$ and such that $X'\cap Y'=P$. Then 
$$A_{X}(P)=A_{X}(X')\cap A_{X}(Y')=f(X')\cap f(Y')=A_{Y}(X')\cap A_{Y}(Y')=A_{Y}(P).$$
If $X\cap Y$ is finite dimensional, then, by Lemma \ref{lemma2-6}, there is $Z\in {\mathcal G}_{\infty}(H)$ such that 
$X\cap Z$ and $Y\cap Z$ are elements of ${\mathcal G}_{\infty}(H)$ containing $X\cap Y$.
Applying the above arguments to $X,Z$ and $Y,Z$, we establish that 
$$A_{X}(P)=A_{Z}(P)=A_{Y}(P).$$

So, we get a mapping $g:{\mathcal G}_{1}(H)\to {\mathcal G}_{1}(H')$. 
This is an isomorphism of $\Pi_{H}$ to $\Pi_{H'}$, hence it is induced by a semilinear isomorphism $A:H\to H'$.
It is clear that 
$$f(X)=A(X)\;\mbox{ for every }\;X\in {\mathcal G}_{\infty}(H)$$ 
and the restriction of $A$ to $X$ is a scalar multiple of $A_{X}$.
Recall that each $A_{X}$ is bounded. Also, $H$ can be presented as the orthogonal sum of two elements from ${\mathcal G}_{\infty}(H)$.
This means that $A$ is bounded, i.e. it is an invertible bounded linear or conjugate-linear operator.

Suppose that $A':H\to H'$ is an invertible bounded linear or conjugate-linear operator
such that  $A'(X)=A(X)$ for every $X\in {\mathcal G}_{\infty}(H)$.
For every $1$-dimensional subspace $P\subset H$ there exist $X,Y\in {\mathcal G}_{\infty}(H)$ satisfying $X\cap Y=P$
and we have
$$A(P)=A(X)\cap A(Y)=A'(X)\cap A'(Y)=A'(P).$$
This implies that $A'$ is a scalar multiple of $A$,
in other words, the extension of $f$ to a lattice isomorphism is unique. 

\begin{rem}{\rm
This proof is an essential modification of the proof given in \cite{Pankov2,Pankov3}.
}\end{rem}

\section{Logics associated to Hilbert spaces}

\subsection{Logics}
A {\it logic} is a bounded lattice $(L,\le)$ together with an {\it orthogonal complementation} ({\it negation}) 
$x\to x^{\perp}$ satisfying the following conditions: 
\begin{enumerate}
\item[(1)] $x\le y$ implies that $y^{\perp} \le x^{\perp}$ for all $x,y\in L$,
\item[(2)] $x^{\perp\perp}=x$ and $x\wedge x^{\perp}=0$ for every $x\in L$,
\item[(3)] if $x,y\in L$ and $x\le y$, then $y=x\vee (x^{\perp}\wedge y)$.
\end{enumerate}
Elements of the logic $L$ are called {\it propositions}.

By the axiom (2), the orthogonal complementation is a bijective transformation of $L$.
It follows from (1) that $0^{\perp}=1$ and $1^{\perp}=0$. The axiom (3) implies that $x\vee x^{\perp}=1$.
The readers can check that De Morgan Law holds true:
$$(x \wedge y)^{\perp}=x^{\perp}\vee y^{\perp}\;\mbox{ and }\;(x\vee y)^{\perp}=x^{\perp}\wedge y^{\perp}$$
for all $x,y\in L$.

We say that $x\in L$ is {\it orthogonal} to $y\in L$ and write $x\perp y$ if $x\le y^{\perp}$.
By the axioms (1) and (2), the latter implies that $y\le x^{\perp}$. 
Therefore, the orthogonality relation is symmetric.

Two distinct propositions $x,y\in L$ are called {\it compatible} if there exist mutually orthogonal propositions $x',y',z\in L$ such that
$$x=z\vee x'\;\mbox{ and }\;y=z\vee y'.$$
Note that $0$ and $1$ are compatible to any proposition.

A {\it quantum logic} is a logic with at least two non-compatible propositions 
and a {\it classical logic} is a logic, where any two propositions are compatible.
See \cite{Cohen,HB-QL,Var} for more information.

For every complex Hilbert space $H$ the lattice ${\mathcal L}(H)$ together with the orthogonal complementation $X\to X^{\perp}$ is a quantum logic.
This logic is known as the {\it standard quantum logic} associated to $H$.

\subsection{Automorphisms of the logic ${\mathcal L}(H)$}
Let $H$ be a complex Hilbert space.
An automorphism of the logic ${\mathcal L}(H)$ is a lattice automorphism $f$ 
commuting with the orthogonal complementation, i.e. 
$$f(X^{\perp})=f(X)^{\perp}\;\mbox{ for every }\;X\in {\mathcal L}(H).$$
Every unitary or anti-unitary operator $A:H\to H$ induces an automorphism of the logic ${\mathcal L}(H)$.
Any non-zero scalar multiple $aA$ induces the same automorphism, 
but the operator $aA$ is unitary or anti-unitary only in the case when the complex scalar $a$ is unit, i.e. $|a|=1$.

The description of automorphisms of the logic is a simple consequence of results from the previous section and the following statement.

\begin{lemma}\label{lemma3-0}
If $\dim H\ge 3$ and $L$ is a semilinear automorphism of $H$ transferring orthogonal vectors to orthogonal vectors, 
i.e. $x\perp y$ implies that $L(x)\perp L(y)$,
then $L$ is a non-zero scalar multiple of an unitary or anti-unitary operator.
\end{lemma}

\begin{proof}
First, we consider the case when $L$ is linear or conjugate-linear.
Since $L$ sends orthogonal vectors to orthogonal vectors, 
for every orthonormal basis $\{e_{i}\}_{i\in I}$ of $H$ there is an orthonormal basis $\{e'_{i}\}_{i\in I}$  such that 
$L(e_{i})=a_{i}e'_{i}$ for a certain non-zero scalar $a_{i}\in {\mathbb C}$.
We can assume that $e_{i}=e'_{i}$ for every $i$ 
(otherwise, we take the operator $UL$, where $U$ is the unitary operator sending every $e'_{i}$ to $e_{i}$).
If $i$ and $j$ are distinct elements of $I$, then the vectors $e_{i}+e_{j}$ and $e_{i}-e_{j}$ are orthogonal and the same holds for the vectors 
$$L(e_{i}+e_{j})=a_{i}e_{i}+a_{j}e_{j}\;\mbox{ and }\;L(e_{i}-e_{j})=a_{i}e_{i}-a_{j}e_{j}.$$
The latter implies that $|a_{i}|=|a_{j}|$ for any pair $i,j\in I$.
In other words, there is a positive real number $b$ such that $a_{i}=bb_{i}$, where $b_{i}$ is a unit complex number.
The linear operator transferring every $e_{i}$ to $b_{i}e_{i}$ is unitary and 
the conjugate-linear operator satisfying the same condition is anti-unitary.
We have $L=bL'$, where $L'$ is one of these operators. 

Now, we need to show that $L$ is linear or conjugate-linear.
If $x,y\in H$ are orthogonal unit vectors, then we apply the above arguments to the orthogonal pair $x+y,x-y$
and establish that 
$$||L(x)||=||L(y)||.$$
If unit vectors $x,y\in H$ are non-orthogonal, then we choose a unit vector $z$ orthogonal to both $x,y$ 
(this is possible, since $\dim H\ge 3$) and get 
$$||L(x)||=||L(z)||=||L(y)||.$$
So, the function $x\to ||L(x)||$ is constant on the set of unit vectors which implies that $L$ is bounded.
Hence $L$ is linear or conjugate-linear.
\end{proof}

\begin{rem}{\rm
We do not prove the statement for $\dim V=2$.
This case is left as an exercise for the readers.
}\end{rem}

\begin{theorem}\label{theorem3-1}
If $\dim H\ge 3$, then every automorphism of the logic ${\mathcal L}(H)$ is induced by an unitary or anti-unitary operator on $H$.
This operator is unique up to a unit scalar multiple. 
\end{theorem}

\begin{proof}
Every automorphism of the logic ${\mathcal L}(H)$ is induced by a semilinear automorphism of $H$.
This follows from Corollary \ref{cor-FTPG} and Theorem \ref{theorem-mackey} for the finite-dimensional and infinite-dimensional case, respectively.
Such a semilinear automorphism sends orthogonal vectors to orthogonal vectors and Lemma \ref{lemma3-0} implies that 
it is a scalar multiple of an unitary or anti-unitary operator.
The second statement is obvious. 
\end{proof}

The above statement fails for the case when $\dim H=2$.

\begin{exmp}\label{exmp-orth2}{\rm
Suppose that $\dim H=2$. The orthogonal complement of every $P\in {\mathcal G}_{1}(H)$ also belongs to ${\mathcal G}_{1}(H)$.
Denote by ${\mathcal X}$ the set of all such pairs $\{P,P^{\perp}\}$.
Every bijective transformation of ${\mathcal X}$ can be extended to an automorphism of the logic,
but such an extension is not unique. 
If $f$ is a bijective transformation of ${\mathcal X}$ sending $\{P,P^{\perp}\}$ to $\{Q,Q^{\perp}\}$,
then we define $g(P)$ as one of the elements from $\{Q,Q^{\perp}\}$ and $g(P^{\perp})$ as the other.
We get a bijective transformation $g$ of ${\mathcal G}_{1}(H)$ preserving the orthogonality relation in both directions.
This bijection can be uniquely extended to an automorphism of the logic.
}\end{exmp}

\begin{theorem}\label{theorem3-2}
Let $f$ be a bijective transformation of ${\mathcal L}(H)$ preserving the orthogonality relation in  both directions, i.e.
for all $X,Y\in {\mathcal L}(H)$ we have
$$X\perp Y\;\Longleftrightarrow\; f(X)\perp f(Y).$$
Then $f$ is an automorphism of the logic ${\mathcal L}(H)$.
\end{theorem}

\begin{proof}
For every $X\in{\mathcal L}(H)$ we denote by ${\rm ort}(X)$ the set of all elements from ${\mathcal L}(H)$ orthogonal to $X$.
Then 
$$f({\rm ort}(X))={\rm ort}(f(X)).$$
If $X,Y\in {\mathcal L}(H)$, then 
$$X\subset Y\Leftrightarrow {\rm ort}(Y)\subset{\rm ort}(X)\Leftrightarrow {\rm ort}(f(Y))\subset{\rm ort}(f(X))
\Leftrightarrow f(X)\subset f(Y).$$
So, $f$ is a lattice automorphism.
Observe that $X^{\perp}$ is the greatest element of ${\rm ort}(X)$.
This implies that $f$ commutes with the orthogonal complementation.
\end{proof}

Using Theorem \ref{theorem-pank1} and Lemma \ref{lemma3-0}, we prove the following.

\begin{theorem}[P. {\v S}emrl \cite{Semrl}]\label{theorem3-3}
If $H$ is infinite-dimensional, then every bijective transformation of ${\mathcal G}_{\infty}(H)$ preserving the orthogonality relation in  both directions
can be uniquely extended to an automorphism of the logic ${\mathcal L}(H)$.
\end{theorem}

\begin{proof}
Let $f$ be a bijective transformation of ${\mathcal G}_{\infty}(H)$ preserving the orthogonality relation in both directions.
As in the proof of Theorem \ref{theorem3-2}, we establish that 
$f$ is an automorphism of the partially ordered set ${\mathcal G}_{\infty}(H)$. 
By Theorem \ref{theorem-pank1}, $f$ is induced by an invertible bounded linear or conjugate-linear operator on $H$.
This operator sends orthogonal vectors to orthogonal vectors and 
Lemma \ref{lemma3-0} implies that it is a scalar multiple of an unitary or anti-unitary operator.
So, $f$ can be extended to an automorphism of the logic ${\mathcal L}(H)$.
It follows from Theorem \ref{theorem-pank1} that such an extension is unique.
\end{proof}

\begin{rem}{\rm
The original proof of the latter statement (see \cite{Semrl}) is not related to Theorem \ref{theorem-pank1}.
}\end{rem}

By classical Wigner's theorem, 
every bijective transformation of ${\mathcal G}_{1}(H)$ preserving the transition probability is induced by an unitary or anti-unitary operator.
Recall that the transition probability between $P,P'\in {\mathcal G}_{1}(H)$ is equal to $|\langle x,x' \rangle|$, 
where $x\in P$ and $x'\in P'$ are unit vectors;
i.e. the transition probability  is zero only in the orthogonal case.
There is a non-bijective analogue of this result (see, for example,  \cite{Geher}).
In the case when $\dim H\ge 3$, the bijective version of Wigner's theorem is contained in the following.

\begin{prop}[U. Uhlhorn \cite{Uhlhorn}]\label{prop3-1}
Every bijective transformation of ${\mathcal G}_{1}(H)$ preserving the orthogonality relation in both directions 
can be uniquely extended to an automorphism of the logic ${\mathcal L}(H)$.
\end{prop}

\begin{proof}
The statement is trivial if $\dim H=2$.
Suppose that $\dim H\ge 3$ and $f$ is a bijective transformation of ${\mathcal G}_{1}(H)$ preserving the orthogonality relation in both directions.
We check that $f$ is an automorphism of the projective space $\Pi_{H}$.

Let $S$ be a $2$-dimensional subspace of $H$.
We take any orthogonal basis $\{e_{i}\}_{\in I}$ for $S^{\perp}$ and 
denote by $P_{i}$ the $1$-dimensional subspace containing $e_{i}$.
All $f(P_{i})$ are mutually orthogonal and we write $S'$ for the maximal closed subspace orthogonal to them.
A $1$-dimensional subspace $P$ is contained in $S$ if and only if $f(P)$ is contained in $S'$.
This implies that $S'$ is $2$-dimensional (since $f$ is orthogonality preserving in both directions).
So, $f$ transfers lines to lines.
Similarly, we show that the same holds for $f^{-1}$.

By Theorem \ref{theorem-FTPG}, $f$ is induced by a semilinear automorphism of $H$.
This semilinear automorphism sends orthogonal vectors to orthogonal vectors, i.e.
it is a scalar multiple of an unitary or anti-unitary operator.
This gives the claim.
\end{proof}

In Section 5, the same statement will be proved for the Grassmannian ${\mathcal G}_{k}(H)$ under the assumption that $\dim H >2k$.
Note that this statement fails if $\dim H=2k\ge 4$. 
In this case, the orthogonal complement of $X\in {\mathcal G}_{k}(H)$ is the unique element of ${\mathcal G}_{k}(H)$ orthogonal to $X$.
As in Example \ref{exmp-orth2},
any bijective transformation of the set of all such pairs $\{X,X^{\perp}\}$ gives a class of bijective transformations of ${\mathcal G}_{k}(H)$
preserving the orthogonality relation in both directions.
Since we assume that $\dim H>2$, every logic automorphism is induced by an unitary or anti-unitary operator.

\begin{rem}\label{rem4-3}{\rm
If $\dim H=2$, then a simple verification shows that every bijective transformation of ${\mathcal G}_{1}(H)$
sending orthogonal elements to orthogonal elements is orthogonality preserving in both directions.
Suppose that $\dim H=n$ is finite and not less than $3$.
Consider a bijective transformation $f$ of ${\mathcal G}_{1}(H)$ which sends orthogonal elements to orthogonal elements.
As in the proof of Proposition \ref{prop3-1}, for a $2$-dimensional subspace $S$
we choose mutually orthogonal $1$-dimensional subspaces $P_{1},\dots,P_{n-2}$ which are orthogonal to $S$.
If $S'$ is the $2$-dimensional subspace orthogonal to all $f(P_{i})$,
then $f$ transfers every $1$-dimensional subspace of $S$ to a $1$-dimensional subspace of $S'$.
So, $f$ sends lines to subsets of lines and, by Remark \ref{rem2-1}, it is an automorphism of $\Pi_{H}$.
Therefore, $f$ is induced by an unitary or anti-unitary operator, i.e. it is orthogonality preserving in both directions.
}\end{rem}

\subsection{Compatibility relation}
Two elements $X,Y\in {\mathcal L}(H)$ are compatible if there exist $X',Y'\in {\mathcal L}(H)$ such that $X\cap Y,X',Y'$ are mutually orthogonal and 
$$X=(X\cap Y)+X',\;\;Y=(X\cap Y)+Y'.$$
It is clear that $X'$ and $Y'$ are the intersections of $(X\cap Y)^{\perp}$ with $X$ and $Y$, respectively.
Therefore, $X$ and $Y$ are compatible if and only if 
$$(X\cap Y)^{\perp}\cap X\;\mbox{ and }\;(X\cap Y)^{\perp}\cap Y$$
are orthogonal.

For example, $X,Y\in {\mathcal L}(H)$ are compatible if $X\subset Y$ or $X\perp Y$, i.e.
the compatibility relation contains the inclusion and orthogonality relations.
Also, if $X\in {\mathcal L}(H)$ is compatible to every element of ${\mathcal L}(H)$, then $X$ is $0$ or $H$.

Every closed subspace $X\subset H$ can be identified with the projection $P_{X}$ whose image is $X$
(Example \ref{exmp-proj}). 
Closed subspaces $X,Y\subset H$ are orthogonal if and only if 
$$P_{X}P_{Y}=P_{Y}P_{X}=0.$$

\begin{prop}\label{prop3-2}
Closed subspaces $X,Y\subset H$ are compatible if and only if the projections $P_{X},P_{Y}$ commute.
\end{prop}

\begin{proof}
Direct verification.
\end{proof}

\begin{rem}{\rm
This simple observation is a partial case of classical von Neumann's theorem.
By the spectral theorem, every bounded self-adjoint operator $A$ on $H$ can be identified with a spectral measure $\mu_{A}$
which takes values in the logic ${\mathcal L}(H)$, or equivalently, in the set of projections on $H$.
Two such measures are called {\it compatible} if all values of one measure are compatible to all values of the other.
The von Neumann theorem states that two bounded self-adjoint operators $A$ and $B$ commute if and only if 
the corresponding measures $\mu_{A}$ and $\mu_{B}$ are compatible.
See \cite{Cohen,Var} for more information.
}\end{rem}

We say that a subset of ${\mathcal L}(H)$ is {\it compatible} if any two distinct elements from this subset are compatible. 
Observe that $X,Y\in {\mathcal L}(H)$ are compatible if and only if there is an orthogonal basis of $H$ such that 
$X$ and $Y$ are spanned by subsets of this basis.
For every orthogonal basis $B$ of $H$ we denote by ${\mathcal A}(B)$ the set formed by all elements of ${\mathcal L}(H)$ spanned by subsets of $B$. 
This is a compatible subset of ${\mathcal L}(H)$.
The equality ${\mathcal A}(B)={\mathcal A}(B')$ implies that the vectors from one of the bases are scalar multiples of the vectors from the other.
So, there is a one-to-one correspondence between subsets of type ${\mathcal A}(B)$ and orthonormal bases of $H$.

\begin{prop}\label{prop3-3}
Every maximal compatible subset of ${\mathcal L}(H)$ is ${\mathcal A}(B)$ for a certain orthogonal basis $B$ of $H$.
\end{prop}

\begin{proof}
We need to show that every compatible subset ${\mathcal X}\subset {\mathcal L}(H)$ is contained in a certain ${\mathcal A}(B)$.
Let ${\mathcal Y}$ be the set formed by all smallest non-zero intersections of elements from ${\mathcal X}$.
In other words, $Y$ belongs to ${\mathcal Y}$ if and only if it is the intersection of some elements from ${\mathcal X}$
and the remaining elements of ${\mathcal X}$ are orthogonal to $Y$. 
Observe that the elements of ${\mathcal Y}$ are mutually orthogonal and $X\in {\mathcal X}$ belongs to ${\mathcal Y}$
only in the case when it is orthogonal to all other elements of ${\mathcal X}$.
For every $X\in {\mathcal X}$ we denote by $M_{X}$ 
the maximal closed subspace in $X$ orthogonal to all elements of ${\mathcal Y}$ contained in $X$.
Consider the set ${\mathcal Z}$ consisting of all elements of ${\mathcal Y}$ and all non-zero $M_{X}$.
The elements of this set are mutually orthogonal and 
there is an orthogonal basis $B$ of $H$ such that every element from ${\mathcal Z}$ is spanned by a subset of $B$. 
It is easy to see that ${\mathcal X}\subset {\mathcal A}(B)$.
\end{proof}

The set ${\mathcal A}(B)$ partially ordered by the inclusion relation $\subset$ together with the orthogonal complementation
is a maximal classical logic contained in the logic ${\mathcal L}(H)$.
Conversely, every maximal classical logic contained in ${\mathcal L}(H)$ is of such type.

\begin{rem}\label{rem-building}{\rm
The logic ${\mathcal L}(H)$ together with all maximal classical sub\-logics ${\mathcal A}(B)$ is a structure 
similar to the buildings associated to general linear groups.
If $X$ is a set (not necessarily finite),
then a {\it simplicial complex} $\Delta$ over $X$ is formed by finite subsets of $X$ such that every one-element subset belongs to $\Delta$
and for every $A\in \Delta$ all subsets of $A$ belong to $\Delta$; 
the elements of $X$ and $\Delta$ are called  {\it vertices} and {\it simplices}, respectively.
For example, if $V$ is a vector space of finite dimension, then the {\it flag complex} ${\mathfrak F}(V)$
is the simplicial complex whose vertices are all subspaces of $V$ and the simplices are all (not necessarily maximal) flags, 
i.e. chains of incident subspaces.
A {\it building} is a simplicial complex together with a family of distinguished subcomplexes called {\it apartments} and satisfying some axioms \cite{Tits}.
The building associated to the group ${\rm GL}(V)$ is the flag complex ${\mathfrak F}(V)$ whose apartments are defined by bases of $V$;
every apartment consists of all flags formed by subspaces spanned by subsets of a certain basis.
The maximal classical sublogics of ${\mathcal L}({\mathbb C}^n)$ correspond to the apartments of ${\mathfrak F}({\mathbb C}^n)$
defined by orthogonal bases.
}\end{rem}

Every automorphism of the logic ${\mathcal L}(H)$ preserves the compatibility relation in both directions.
However, the automorphism group of the logic ${\mathcal L}(H)$ is a proper subgroup in
the group of all bijective transformations of ${\mathcal L}(H)$ preserving the compatibility relation in both directions.
For example, the orthogonal complementation $X\to X^{\perp}$ belongs to this group, but it is not a logic automorphism.
Consider a more general example.

\begin{exmp}{\rm
Let ${\mathcal X}$ be a subset of ${\mathcal L}(H)$ satisfying the following condition:
for every $X\in {\mathcal X}$ the orthogonal complement $X^{\perp}$ belongs to ${\mathcal X}$.
Denote by $\pi_{\mathcal X}$ the bijective transformation of ${\mathcal L}(H)$ defined as follows
$$\pi_{\mathcal X}(X)=\begin{cases}
X^{\perp} &\mbox{if }\;\; X\in {\mathcal X}\\
X&\mbox{if }\;\; X\not\in {\mathcal X}\,.
\end{cases}$$
This transformation preserves the compatibility relation in both directions
(since every $X$ compatible to $Y$ is also compatible to $Y^{\perp}$).
}\end{exmp}

We state that every bijective transformation of ${\mathcal L}(H)$ preserving the compatibility relation in both directions 
is a logic automorphism or the composition of a logic  automorphism and a certain $\pi_{\mathcal X}$.

\begin{theorem}[L. Moln\'ar and P. \v{S}emrl \cite{MolnarSemrl}]\label{theorem-MS}
Let $f$ be a bijective transformation of ${\mathcal L}(H)$ preserving the compatibility relation in both directions, i.e.
$X,Y\in {\mathcal L}(H)$ are compatible if and only if $f(X),f(Y)$ are compatible.
Then there exists an automorphism $g$ of the logic ${\mathcal L}(H)$ such that for every $X\in {\mathcal L}(H)$ we have either 
$$f(X)=g(X)\;\mbox{ or }\;f(X)=g(X)^{\perp}.$$
 \end{theorem}

\begin{rem}{\rm
In \cite{MolnarSemrl}, this result was formulated in terms of commutativity of projections.
}\end{rem}

The same statement holds for the Grassmannian ${\mathcal G}_{\infty}(H)$.

\begin{theorem}[L. Plevnik \cite{Plevnik}]\label{theorem-plevnik}
Suppose that $H$ is infinite-dimensional and $f$ is a bijective transformation of ${\mathcal G}_{\infty}(H)$
preserving the compatibility relation in both directions. 
Then there exists an automorphism $g$ of the logic ${\mathcal L}(H)$ such that for every $X\in {\mathcal G}_{\infty}(H)$ we have either 
$$f(X)=g(X)\;\mbox{ or }\;f(X)=g(X)^{\perp}.$$
\end{theorem}

\begin{rem}{\rm
Recall that ${\mathcal I}_{\infty}(H)$ is the set of all bounded idempotents on $H$
whose image and kernel both are infinite-dimensional (Remark \ref{rem-projections}).
In \cite{Plevnik}, L. Plevnik describes commutativity preserving bijective transformations of ${\mathcal I}_{\infty}(H)$.
We use the same arguments to prove Theorem \ref{theorem-plevnik}.
There is a natural one-to-one correspondence between idempotents and involutions (Example \ref{exmp-inv}) and
the mentioned above Plevnik's result can be reformulated in terms of commutativity preserving transformations of the set of bounded involutions
corresponding the idempotents from ${\mathcal I}_{\infty}(H)$.
In other words, this is an infinite-dimensional version of the description of
commutativity preserving bijective transformations of the sets of conjugate involutions  in the general linear group \cite{Pankov1}.
Earlier,  results of the same nature were exploited to determining automorphisms of classical groups, see \cite{D1,D2}.
}\end{rem}

\subsection{Proof of Theorems \ref{theorem-MS} and \ref{theorem-plevnik}}
For every subset ${\mathcal X}\subset {\mathcal L}(H)$ we denote by ${\mathcal X}^{c}$ the set of all elements of ${\mathcal L}(H)$
compatible to every element of ${\mathcal X}$ and write ${\mathcal X}^{cc}$ instead of $({\mathcal X}^{c})^{c}$.
Note that $0$ and $H$ always belong to ${\mathcal X}^{c}$.

Let $X$ and $Y$ be distinct compatible elements of ${\mathcal L}(H)$ both different from $0$ and $H$.
Then $H$ can be presented as the orthogonal sum of the following subspaces
$$Z_{1}=X\cap Y,\;Z_{2}=X^{\perp}\cap Y,\;Z_{3}=X\cap Y^{\perp},\;Z_{4}=X^{\perp}\cap Y^{\perp}$$
(some of them may be zero) and $\{X,Y\}^{cc}$ consists of all orthogonal sums 
$$\sum_{i\in I}Z_{i}\;\mbox{ with }\;I\subset\{1,2,3,4\}$$
(note that $0$ and $H$ correspond to the cases when $I=\emptyset$ and $I=\{1,2,3,4\}$, respectively). 
Some of these sums may be coincident and we have
$$|\{X,Y\}^{cc}|=2^{k},$$
where $k$ is the number of non-zero $Z_{i}$.

Using the number of elements in $\{X,Y\}^{cc}$, we can distinguish elements from ${\mathcal G}_{1}(H)\cup {\mathcal G}^{1}(H)$.

\begin{lemma}\label{lemma3-1}
If $X\in {\mathcal G}_{1}(H)\cup {\mathcal G}^{1}(H)$, then for every $Y\in{\mathcal L}(H)\setminus\{0,H\}$ compatible to $X$
we have 
$$|\{X,Y\}^{cc}|\in\{4,8\}.$$
In the case when $X\in{\mathcal L}(H)\setminus\{0,H\}$ does not belong to ${\mathcal G}_{1}(H)\cup {\mathcal G}^{1}(H)$, there is $Y\in{\mathcal L}(H)$ 
compatible to $X$ and such that
$$|\{X,Y\}^{cc}|=16.$$
\end{lemma}

\begin{proof}
Suppose that $X$ is an element of ${\mathcal G}_{1}(H)\cup {\mathcal G}^{1}(H)$. 
If $Y$ is the orthogonal complement of $X$, then $Z_{1}=Z_{4}=0$ and $\{X,Y\}^{cc}$ consists of $4$ elements.
For all other cases, there are precisely tree non-zero $Z_{i}$ and the number of elements in $\{X,Y\}^{cc}$ is equal to $8$.

Suppose that $X\in{\mathcal L}(H)\setminus\{0,H\}$ does not belong to ${\mathcal G}_{1}(H)\cup {\mathcal G}^{1}(H)$.
We take any $Y\in {\mathcal L}(H)$ compatible to $X$ and such that $X\cap Y$ is non-zero and $X+Y$ is a proper subspace of $H$.
Then all $Z_{i}$ are non-zero and the number of elements in $\{X,Y\}^{cc}$ is equal to $16$.
\end{proof}

To prove Theorem \ref{theorem-plevnik} we will consider the intersection of  ${\mathcal G}_{\infty}(H)$
with $\{X,Y\}^{cc}$ for $X,Y\in {\mathcal G}_{\infty}(H)$.

\begin{lemma}\label{lemma3-2}
Suppose that $H$ is infinite-dimensional.
If $X,Y$ are distinct compatible elements from ${\mathcal G}_{\infty}(H)$ 
and $Y\ne X^{\perp}$, then the following two conditions are equivalent:
\begin{enumerate}
\item[(1)] $|\{X,Y\}^{cc}\cap {\mathcal G}_{\infty}(H)|\in \{4,6\}$,
\item[(2)] $X\perp Y$ or $X^{\perp}\perp Y$ or $X\perp Y^{\perp}$ or $X^{\perp}\perp Y^{\perp}$.
\end{enumerate}
\end{lemma}

\begin{proof}
The assumption $Y\ne X^{\perp}$ implies that  at least three $Z_{i}$ are non-zero.

The condition (2) is equivalent to the fact that one of $Z_{i}$ is zero, i.e. there are precisely three non-zero $Z_i$.
We have
$$X=Z_{1}+Z_{3},\;X^{\perp}=Z_{2}+Z_{4}\;\mbox{ and }\;Y=Z_{1}+Z_{2},\;Y^{\perp}=Z_{3}+Z_{4}.$$
This guarantees that at most one of non-zero $Z_{i}$ is finite-dimensional
(otherwise, at most one of $Z_{i}$ is infinite-dimensional and $X$ or $Y$ has finite dimension or codimension which is impossible).
If all non-zero $Z_{i}$ are infinite-dimensional, then
$$|\{X,Y\}^{cc}\cap {\mathcal G}_{\infty}(H)|=6$$
(each non-zero $Z_{i}$ and the sum of any two non-zero $Z_i$ are elements of ${\mathcal G}_{\infty}(H)$).
If one of non-zero $Z_i$ is finite-dimensional, 
then it and its orthogonal complement do not belong to ${\mathcal G}_{\infty}(H)$ and we have
$$|\{X,Y\}^{cc}\cap {\mathcal G}_{\infty}(H)|=4.$$

If (2) fails, then all $Z_{i}$ are non-zero.
In this case, there are at most two finite-dimensional $Z_{i}$  
(otherwise, at most one of $Z_{i}$ is infinite-dimensional and $X$ or $Y$ has finite dimension or codimension).
If all $Z_{i}$ are infinite-dimensional, then every orthogonal sum $\sum_{i\in I}Z_{i}$, where $I$ is a proper subset of $\{1,2,3,4\}$,
belongs to ${\mathcal G}_{\infty}(H)$ and 
$$|\{X,Y\}^{cc}\cap {\mathcal G}_{\infty}(H)|=14.$$
If only one of $Z_{i}$ is finite-dimensional, then it and its orthogonal complement do not belong to ${\mathcal G}_{\infty}(H)$ and we have
$$|\{X,Y\}^{cc}\cap {\mathcal G}_{\infty}(H)|=12.$$
Similarly, if $Z_{i}$ and $Z_{j}$ are finite-dimensional, 
then $Z_{i},Z_{j},Z_{i}+Z_{j}$ ant their orthogonal complements do not belong to ${\mathcal G}_{\infty}(H)$
which means that
$$|\{X,Y\}^{cc}\cap {\mathcal G}_{\infty}(H)|=8.$$
We get the claim.
\end{proof}

\begin{proof}[Proof of Theorem \ref{theorem-MS}]
Let $f$ is a bijective transformation of ${\mathcal L}(H)$ preserving the compatibility relation in both directions.
If a closed subspace is compatible to all elements of ${\mathcal L}(H)$, then this subspace is $0$ or $H$.
Therefore, $f$ transfers the set $\{0,H\}$ to itself.
Lemma \ref{lemma3-1} implies that $f$ also sends ${\mathcal G}_{1}(H)\cup {\mathcal G}^{1}(H)$ to itself.
The case when $\dim H=2$ is trivial and we suppose that $\dim H\ge 3$.

Consider the bijective transformation $h$ of ${\mathcal G}_{1}(H)$ defined as follows
$$h(X)=\begin{cases}
f(X)&\mbox{if }\;\; f(X)\in {\mathcal G}_{1}(H)\\
f(X)^{\perp}&\mbox{if }\;\; f(X)\in {\mathcal G}^{1}(H)\,.
\end{cases}$$
Then $h$ preserves the orthogonality relation in both directions and, by Proposition \ref{prop3-1},
it can be extended to a certain automorphism $g$ of the logic ${\mathcal L}(H)$.
The transformation $g^{-1}f$ preserves the compatibility relation in both directions.
Also, $g^{-1}f$ leaves fixed every $X\in {\mathcal G}_{1}(H)$ or sends it to the orthogonal complement $X^{\perp}$.
Therefore, $Y\in {\mathcal L}(H)$ is compatible to $X\in {\mathcal G}_{1}(H)$ if and only if $g^{-1}f(Y)$ is compatible to $X$
(any element of ${\mathcal L}(H)$  is compatible to $X$ if and only if it is compatible to $X^{\perp}$).
A $1$-dimensional subspace is compatible to $Y$ if and only if it is contained in $Y$ or $Y^{\perp}$.
Therefore, $g^{-1}f(Y)$ coincides with $Y$ or $Y^{\perp}$ for every $Y\in {\mathcal L}(H)$.
So, the logic automorphism $g$ is as required. 
\end{proof}

\begin{proof}[Proof of Theorem \ref{theorem-plevnik}]
Suppose that $H$ is infinite-dimensional and
$f$ is a bijective transformation of ${\mathcal G}_{\infty}(H)$ preserving the compatibility relation in both directions.
For distinct compatible elements $X,Y\in {\mathcal G}_{\infty}(H)$ we have 
$$|\{X,Y\}^{cc}\cap {\mathcal G}_{\infty}(H)|=2$$
if and only if $Y$ is the orthogonal complement of $X$. 
Therefore, $f$ preserves the orthogonal complementation, i.e. 
$$f(X^{\perp})=f(X)^{\perp}$$ 
for every $X\in {\mathcal G}_{\infty}(H)$.

We will construct a bijective transformation $g$ of ${\mathcal G}_{\infty}(H)$ satisfying the follo\-wing conditions:
\begin{enumerate}
\item[$\bullet$] $g$ is orthogonality preserving in both directions,
\item[$\bullet$] for every $X\in {\mathcal G}_{\infty}(H)$ we have $g(X)=f(X)$ or $g(X)=f(X)^{\perp}$.
\end{enumerate}
Theorem \ref{theorem3-3} states that $g$ can be extended to a logic automorphism and we get the claim.

Let $X\in {\mathcal G}_{\infty}(H)$. We take any $Y\in {\mathcal G}_{\infty}(H)$ orthogonal to $X$ and distinct from $X^{\perp}$.
By Lemma \ref{lemma3-2}, one of the following two possibilities is realized:
\begin{enumerate}
\item[(1)] $f(X)$ is orthogonal to $f(Y)$ or $f(Y)^{\perp}$,
\item[(2)] $f(X)^{\perp}$ is orthogonal to $f(Y)$ or $f(Y)^{\perp}$.
\end{enumerate}
In the first case, we set $g(X)=f(X)$.
In the second case, we define $g(X)$ as the orthogonal complement of $f(X)$.
We need to show that the definition of $g(X)$ does not depend on the choice of element $Y\ne X^{\perp}$ orthogonal to $X$.
In other words, if the possibility $(i)$, $i\in\{1,2\}$ is realized for a certain $Y\in {\mathcal G}_{\infty}(H)\setminus\{X^{\perp}\}$ orthogonal to $X$,
then the same possibility is realized for all such $Y$.

Let $Y$ and $Z$ be distinct elements of ${\mathcal G}_{\infty}(H)$ orthogonal to $X$ and distinct from $X^{\perp}$. 
First, we consider the case when $Y$ and $Z$ are non-compatible.
Suppose that $f(X)$ is orthogonal to $f(Y)$ or $f(Y)^{\perp}$ and $f(X)^{\perp}$ is orthogonal to $f(Z)$ or $f(Z)^{\perp}$.
In other words, one of $f(Y),f(Y)^{\perp}$ is contained in $f(X)^{\perp}$
and one of $f(Z),f(Z)^{\perp}$ is contained in $f(X)$.
This means that one of $f(Y),f(Y)^{\perp}$ is compatible to one of $f(Z),f(Z)^{\perp}$.
Since $f$ is compatibility preserving in both directions, 
one of $Y,Y^{\perp}$ is compatible to one of $Z,Z^{\perp}$ which contradicts the fact that $Y$ and $Z$ are non-compatible.
Therefore, $f(X)$ is orthogonal to $f(Y)$ or $f(Y)^{\perp}$ if and only if it is orthogonal to $f(Z)$ or $f(Z)^{\perp}$.

In the case when $Y$ and $Z$ are compatible, we choose any $Y'\in {\mathcal G}_{\infty}(H)\setminus\{X^{\perp}\}$ orthogonal to $X$ 
and non-compatible to both $Y$ and $Z$.
By the arguments given above, the following three conditions are equivalent:
\begin{enumerate}
\item[$\bullet$] $f(X)$ is orthogonal to $f(Y)$ or $f(Y)^{\perp}$,
\item[$\bullet$] $f(X)$ is orthogonal to $f(Y')$ or $f(Y')^{\perp}$,
\item[$\bullet$] $f(X)$ is orthogonal to $f(Z)$ or $f(Z)^{\perp}$.
\end{enumerate}
So, the transformation $g$ is well-defined.

Since $f$ is bijective, we have $g(X)\ne g(Y)$ in the case when $Y\ne X^{\perp}$.
It is easy to see that $g(X^{\perp})$ coincides with $g(X)$ or $g(X)^{\perp}$,
but we cannot state that $g(X)\ne g(X^{\perp})$ at this moment.
We need to show that $g$ is bijective.
Let us consider the inverse transformation $f^{-1}$ and the associated transformation $g'$ defined as $g$ for $f$.

Let $X\in {\mathcal G}_{\infty}(H)$ and $X'=g(X)$. Suppose that $g(X)=f(X)$.
Then for every $Y\in {\mathcal G}_{\infty}(H)\setminus \{X^{\perp}\}$ orthogonal to $X$ we have 
$$f(X)\perp f(Y)\;\mbox{ or }\;f(X)\perp f(Y)^{\perp}.$$
If $f(X)$ is orthogonal to $f(Y)$, then we consider $Y'=f(Y)$ which is orthogonal to $X'$ and distinct from $X'^{\perp}$.
Since $f^{-1}(X')=X$ and $f^{-1}(Y')=Y$ are orthogonal, we have $g'(X')=X$.
In the case when $f(X)$ is orthogonal to $f(Y)^{\perp}$, we take $Y'=f(Y)^{\perp}$.
As above, $Y'$ is orthogonal to $X'$ and distinct from $X'^{\perp}$.
Since $f^{-1}(X')=X$ is orthogonal to $f^{-1}(Y')^{\perp}=Y$, we get $g'(X')=X$ again.

Now, we suppose that $g(X)=f(X)^{\perp}$.
Then for every $Y\in {\mathcal G}_{\infty}(H)\setminus \{X^{\perp}\}$ orthogonal to $X$ we have 
$$f(X)^{\perp}\perp f(Y)\;\mbox{ or }\;f(X)^{\perp}\perp f(Y)^{\perp}.$$
Consider the first possibility (the second is similar).
In this case, $Y'=f(Y)$ is orthogonal to $X'=f(X)^{\perp}$ and distinct from $X'^{\perp}$.
Then $f^{-1}(X')^{\perp}=X$ and $f^{-1}(Y')=Y$ are orthogonal. 
This means that $g'(X')=X$.

So, for every $X\in {\mathcal G}_{\infty}(H)$ we have $g'g(X)=X$ and the same arguments show that $gg'(X)=X$.
Therefore, $g$ is bijective which guarantees that 
$$g(X^{\perp})=g(X)^{\perp}$$
for every $X\in {\mathcal G}_{\infty}(H)$. 
To complete the proof we need to establish that $g$ is orthogonality preserving in both directions.

Suppose that $X,Y\in {\mathcal G}_{\infty}(H)$ are orthogonal and $Y\ne X^{\perp}$.
If $g(X)=f(X)$, then $f(X)$ is orthogonal to $f(Y)$ or $f(Y)^{\perp}$ and 
$$g(Y)=f(Y)\;\mbox{ or }\;g(Y)=f(Y)^{\perp},$$
respectively. For each of these cases we have $g(X)\perp g(Y)$.
The case when $g(X)=f(X)^{\perp}$ is similar.
Therefore, $g$ sends orthogonal pairs to orthogonal pairs. 
It is clear that the same holds for $g'=g^{-1}$ and $g$ is orthogonality preserving in both directions.
\end{proof}

\subsection{Kakutani-Mackey theorem}
Let $V$ be a complex normed space.
Denote by ${\mathcal L}_{c}(V)$ the associated lattice of closed subspaces,
i.e. the set of all closed subspaces of $V$ partially ordered by the inclusion relation $\subseteq$.
As in the lattices of closed subspaces of Hilbert spaces,
for any two closed subspaces $X,Y\subset V$ the greatest lower bound is the intersection $X\cap Y$ and 
the least upper bound is the minimal closed subspace containing $X+Y$.
The lattice ${\mathcal L}_{c}(V)$ is bounded: the least element is $0$ and the greatest element is $V$.
 
The direct analogue of Theorem \ref{theorem-mackey} holds for 
the lattices of closed subspaces of infinite-dimensional normed vector spaces.
The proof is similar, but we need some additional arguments which hold automatically for Hilbert spaces.

The following classical result says that for an infinite-dimensional complex Banach space $V$ the lattice ${\mathcal L}_c(V)$
together with an orthogonal complementation satisfying the logic axioms (1) and (2) is the standard quantum logic.

\begin{theorem}[S. Kakutani and G. W. Mackey \cite{KakutaniMackey}]\label{theorem-km}
Let $V$ be an infinite-dimen\-sional complex Banach space.
Suppose that there is a bijective transformation $X\to X^{\perp}$ of the lattice ${\mathcal L}_c(V)$
satisfying the following conditions:
\begin{enumerate}
\item[(1)] for any $X,Y\in {\mathcal L}_{c}(V)$ the inclusion
$X\subset Y$ implies that $Y^{\perp}\subset X^{\perp}$,
\item[(2)] $X^{\perp\perp}=X$ and $X\cap X^{\perp}=0$ for every $X\in {\mathcal L}_{c}(V)$.
\end{enumerate}
Then there is an inner product $V\times V\to {\mathbb C}$ such that the following assertions are fulfilled:
\begin{enumerate}
 \item[$\bullet$] The vector space $V$ together with this inner product
is a complex Hilbert space.
\item[$\bullet$] The identity transformation of $V$ is an invertible bounded linear operator
of the Banach space to the Hilbert space, 
i.e. a subspace of $V$ is closed in the Banach space if and only if it is closed in the Hilbert space\footnote{We cannot state that 
the norm related to the inner product coincides with the primordial norm, but these norms define the same topology on $V$.}.
\item[$\bullet$] For every $X\in {\mathcal L}_{c}(V)$ the subspace $X^{\perp}$ is the orthogonal complement of $X$ in the Hilbert space. 
\end{enumerate}
In other word, the lattice ${\mathcal L}_{c}(V)$ together with this orthogonal complementation is the standard quantum logic. 
\end{theorem}

\begin{proof}[Sketch of proof]
Let $V^{*}$ be the vector space formed by all bounded linear functionals on $V$.
This is a normed vector space:
the norm of $l\in V^{*}$ is the smallest number $a$ such that 
$$||l(x)||\le a||x||$$ for all vectors $x\in V$.
For every $X\in {\mathcal L}_{c}(V)$ we denote by $X^{0}$ the annihilator of $X^{\perp}$ in $V^{*}$, 
i.e. the  set of all bounded linear functionals $l\in V^{*}$ satisfying $l(X^{\perp})=0$.
This is a closed subspace of $V^{*}$.

Using (1), we establish that the bijection of ${\mathcal G}_{1}(V)$ to ${\mathcal G}_{1}(V^{*})$
sending every $X$ to $X^{0}$ is an isomorphism of $\Pi_{V}$ to $\Pi_{V^{*}}$.
By Theorem \ref{theorem-FTPG}, there is a semilinear isomorphism $L:V\to V^{*}$  such that 
$$L(X)=X^{0}$$
for every $X\in {\mathcal G}_{1}(V)$.
It is not difficult to see that the same equality holds for all $X\in {\mathcal L}_{c}(V)$.
In particular, $L$ sends closed subspaces of codimension $1$ to closed subspaces. 
There is the direct analogue of Lemma \ref{lemma2-4} for infinite-dimensional normed vector spaces \cite[Lemma 2]{KakutaniMackey}.
Therefore, $L$ is a linear or conjugate-linear invertible bounded operator. 

Suppose that $L$ is linear.
Let $x$ and $y$ be linearly independent vectors of $V$.
We set $l=L(x)$ and $s=L(y)$. Then 
$$L(x+ay)=l+as.$$
By (2), we have $X\cap X^{\perp}=0$ for every $X\in {\mathcal L}_{c}(X)$.
This implies that each of the scalars 
$$l(x),\;s(y),\;(l+as)(x+ay)$$ is non-zero.
On the other hand, we have
$$(l+as)(x+ay)=l(x)+a(l(y)+s(x))+a^2s(y)$$
and the equation 
$$l(x)+a(l(y)+s(x))+a^2s(y)=0$$ has a solution for $a$.
We get a contradiction which means that $L$ is conjugate-linear.

Now, we define the inner product on $V$. For all vectors $x,y\in V$ we set 
$$\langle x,y \rangle=l(x),\;\mbox{ where }\;l=L(y).$$
The condition (2) guarantees that $\langle x,x\rangle$ is non-zero for every non-zero vector $x\in V$.
Since $L$ is unique up to a non-zero scalar multiple, we can assume that for a certain vector $x_{0}\in V$
the scalar $\langle x_{0},x_{0}\rangle$ is a positive real number.

It clear that $x\to \langle x,y\rangle$ is linear and $x\to \langle y,x\rangle$ is conjugate-linear for every fixed $y\in V$.
We need to show that 
\begin{equation}\label{eq3-1}
\langle x,y\rangle=\overline{\langle y,x\rangle}
\end{equation}
for all $x,y\in X$ and $\langle x,x\rangle$ is a positive real number for every non-zero $x\in V$.

It easily follows from (1) that for any two vectors $x,y\in V$ we have $\langle x,y\rangle =0$ if and only if $\langle y,x\rangle=0$.
Suppose that $\langle x,y \rangle$ is non-zero. We choose non-zero scalars $a,b\in {\mathbb C}$ such that 
\begin{equation}\label{eq-km1}
a\langle x,x\rangle + \langle x,y\rangle =0=b\langle y,y\rangle + \langle x,y\rangle.
\end{equation} 
Then $\langle x, \overline{a}x+y\rangle=0$ which implies that $\langle \overline{a}x+y, x\rangle=0$ and 
$\overline{\langle \overline{a}x+y, x\rangle}=0$, i.e. 
\begin{equation}\label{eq-km2}
a\overline{\langle x,x \rangle} + \overline{\langle y,x\rangle}=0.
\end{equation}
Similarly, we obtain that 
\begin{equation}\label{eq-km3}
b\overline{\langle y,y \rangle} + \overline{\langle y,x\rangle}=0.
\end{equation}
Using \eqref{eq-km1}--\eqref{eq-km3}, we establish that 
$$\overline{\langle x,x \rangle}:\langle x,x \rangle=\overline{\langle y,y \rangle}:\langle y,y \rangle.$$
In other words, for any two vectors $x,y\in V$ satisfying $\langle x,y\rangle\ne 0$ 
the scalar $\langle x,x\rangle$ is real if and only if $\langle y,y\rangle$ is real.
Recall that there is non-zero $x_{0}\in V$ such that $\langle x_{0},x_{0}\rangle$ is real.
Then $\langle x,x\rangle$ is real if $\langle x_{0},x\rangle$ is non-zero. 
In the case when $\langle x_{0},x\rangle =0$, we take any vector $y\in V$ such that $\langle x_{0},y\rangle$ and $\langle x,y\rangle$ both are non-zero.
So,  for every non-zero vector $x\in V$ the scalar $\langle x,x\rangle$ is a non-zero real number.
Then \eqref{eq-km1} and \eqref{eq-km2} imply \eqref{eq3-1}.
For every non-zero $x\in V$ we consider the real function 
$$h(t)=\langle tx+(1-t)x_{0},tx+(1-t)x_{0}\rangle$$
defined on the segment $[0;1]$. 
The function is continuos (the operator $L$ is bounded) and $h(t)$ is non-zero for every $t\in [0;1]$. Since $h(0)>0$, we have always $h(t)>0$.

Therefore, $\langle\cdot,\cdot \rangle$ is an inner product on $V$.
Using the fact that $L$ is bounded, the readers can show that 
the identity transformation of $V$ is an invertible bounded linear operator
of the Banach space to the normed vector space related to the inner product $\langle\cdot,\cdot \rangle$.
This implies that the norm defined by the inner product is complete.
\end{proof}

%%%%%%%%%%%%%%%%%%%%%%%%%%%%%%%%%%%%%%%%%%%%%%%%%%%%%%%%%%%%%%%%%%%%%%%%

\section{Grassmannians of vector spaces}

In this section, we consider the transformations of Grassmannians of vector spaces induced by semilinear isomorphisms and 
present some characterizations of such transformations.
One of them is well-known Chow's theorem \cite{Chow}. We will need it in the next section.
Also, these transformations can be characterized as apartments preserving. 
This result is not exploited in what follows,  but we will use the same idea to study compatibility preserving transformations.

\subsection{Chow's theorem}
Let $V$ be a vector space over a field. 
Recall that for every natural $k<\dim V$ we denote by ${\mathcal G}_{k}(V)$ the Grassmannian formed by $k$-dimensional subspaces of $V$.
Two $k$-dimensional subspaces of $V$ are called {\it adjacent} if their intersection is $(k-1)$-dimensional.
This is equivalent to the fact that the sum of these subspaces is $(k+1)$-dimensional.
Any two distinct $1$-dimensional subspaces of $V$ are adjacent.
If $\dim V=n$ is finite, then the same holds for any two distinct $(n-1)$-dimensional subspaces of $V$.

The {\it Grassmann graph} $\Gamma_{k}(V)$ is the graph whose vertex set is the Grassmannian ${\mathcal G}_{k}(V)$
and whose edges are pairs of adjacent $k$-dimensional subspaces.
This graph is connected, i.e. for any $X,Y\in {\mathcal G}_{k}(V)$ there is a sequence 
$$X=X_{0},X_{1},\dots,X_{i}=X,$$
where $X_{j-1}$ and $X_{j}$ are adjacent elements of ${\mathcal G}_{k}(V)$ for every $j\in \{1,\dots,i\}$.
The smallest number $i$ for which such a sequence exists is called the {\it path distance} between $X$ and $Y$ (see, for example, \cite[Section 15.1]{DD}), 
we denote this distance  by $d(X,Y)$.
It is not difficult to prove that 
$$d(X,Y)=k-\dim(X\cap Y)=\dim(X+Y)-k.$$
Every semilinear automorphism of $V$ induces an automorphism of the Grassmann graph $\Gamma_{k}(V)$.

Denote by $V^{*}$ the dual vector space formed by all linear functionals on $V$.
If $V$ is finite-dimensional, then $\dim V=\dim V^{*}$ and the second dual vector space $V^{**}$ can be naturally identified with $V$.
In the case when $V$ is infinite-dimensional, we have $\dim V<\dim V^{*}$, see \cite[Section II.3]{Baer}.
For every subset $X\subset V$  the {\it annihilator} 
$$X^{0}=\{x^{*}\in V^{*}: x^{*}(x)=0\mbox{ for all }x\in X\}$$
is a subspace of $V^{*}$.
Similarly, for every subset $Y\subset V^{*}$ the {\it annihilator}
$$Y^{0}=\{x\in V: x^{*}(x)=0\mbox{ for all }x^{*}\in Y\}$$
is a subspace of $V$.

Suppose that $\dim V=n$ is finite. 
For every subspace of $V$ or $V^{*}$ the dimension of the annihilator is equal to the codimension of this subspace
and the annihilator of the annihilator coincides with the subspace.
The annihilator mapping $X\to X^{0}$ is a bijection of ${\mathcal L}(V)$ to ${\mathcal L}(V^{*})$ reversing the inclusion relation 
and sending every ${\mathcal G}_{k}(V)$ to ${\mathcal G}_{n-k}(V^{*})$.

\begin{prop}\label{prop-ann}
If $\dim V=n$ is finite, then the annihilator mapping
induces an isomorphism between the Grassmann graphs $\Gamma_{k}(V)$ and $\Gamma_{n-k}(V^{*})$.
\end{prop}

\begin{proof}
Easy verification.
\end{proof}

Every semilinear isomorphism of $V$ to $V^{*}$ induces an isomorphism of  $\Gamma_{k}(V)$ to $\Gamma_{k}(V^{*})$. 
The composition of this isomorphism and the annihilator mapping is an isomorphism of $\Gamma_{k}(V)$ to $\Gamma_{n-k}(V)$
and we get an automorphism of $\Gamma_{k}(V)$ if $n=2k$.

\begin{theorem}[W.L. Chow \cite{Chow}]\label{theorem-chow}
Suppose that $k>1$ and, in addition, we require that $k<\dim V-1$ if $V$ is finite-dimensional.
Then every automorphism of the Grassmann graph $\Gamma_{k}(V)$ is induced by a semilinear automorphism of $V$
or a semilinear isomorphism of $V$ to $V^{*}$ and the second possibility is realized only in the case when $\dim V=2k$.
\end{theorem}

If $k=1$ or $V$ is finite-dimensional and $k=\dim V-1$, then
any two distinct elements of ${\mathcal G}_{k}(V)$ are adjacent and every bijective transformation of ${\mathcal G}_{k}(V)$
is an automorphism of the graph $\Gamma_{k}(V)$.

In \cite{Chow} (see also \cite{D2,Pankov-book1,Wan}), this statement was proved only for the case when $V$ is finite-dimensional,
but the same arguments work if $V$ is infinite-dimensional.
Also, classical Chow's theorem follows immediately from the description of isometric embeddings of Grassmann graphs \cite[Chapter 3]{Pankov-book2}.
For these reasons, we only sketch the proof of Theorem \ref{theorem-chow}.
It is based on the description of maximal cliques in the Grassmann graph $\Gamma_{k}(V)$.

Recall that a subset in the vertex set of a graph is called a {\it clique} if any two distinct elements of this subset are adjacent vertices in the graph.

From this moment, we suppose that $k>1$ and, in addition, $k<\dim V-1$ if $V$ is finite-dimensional.
For every subspace $S\subset V$ we denote by $[S\rangle_{k}$ the set of all $k$-dimensional subspaces containing $S$.
If $S$ is $(k-1)$-dimensional, then $[S\rangle_{k}$ is a clique of $\Gamma_{k}(V)$. 
Cliques of such type are called {\it stars}.
For every subspace $U\subset V$ we write $\langle U]_{k}$ for the set of all $k$-dimensional subspaces contained in $U$.
In the case when $U$ is $(k+1)$-dimensional, this is a clique of $\Gamma_{k}(V)$. 
Every such clique is said to be  a {\it top}.

\begin{prop}\label{prop-cliqueGr}
Every maximal clique of $\Gamma_{k}(V)$ is a star or a top.
\end{prop}

\begin{proof}[Sketch of proof]
It is sufficiently to show that every clique ${\mathcal C}$ of the graph $\Gamma_{k}(V)$
is contained in a star or a top.
The statement is trivial if ${\mathcal C}$ consists of two elements.
Suppose that $|{\mathcal C}|\ge 3$.
If $X,Y\in {\mathcal C}$ and ${\mathcal C}$ is not contained in the star $[X\cap Y\rangle_k$,
then we show that it is a subset of the top $\langle X+Y]_k$.
\end{proof}

To prove Theorem \ref{theorem-chow} we will use the following intersection properties of maximal cliques.
The intersection of two distinct stars of ${\mathcal G}_{k}(V)$ is empty or it contains precisely one element,
the second possibility is realized if and only if the associated $(k-1)$-dimensional subspaces are adjacent.
Similarly, the intersection of two distinct tops of ${\mathcal G}_{k}(V)$ is empty or a one-element set
and the second possibility is realized only in the case when the associated $(k+1)$-dimensional subspaces are adjacent.
The intersection of a star $[S\rangle_{k}$ and a top $\langle U]_{k}$ is non-empty if and only if $S$ is contained in $U$.
Every such intersection is called a {\it line} of ${\mathcal G}_{k}(V)$.
A star $[S\rangle_k$ and a top $\langle U]_k$ together with all lines contained in them can be identified with 
the projective spaces associated to $V/S$ and $U^{*}$, respectively.

\begin{proof}[Proof of Theorem \ref{theorem-chow} (sketch)]
Let $f$ be an automorphism of the Grassmann graph $\Gamma_{k}(V)$. 
Then $f$ and $f^{-1}$ transfer maximal cliques (stars and tops) to maximal cliques.
Since the intersection of two distinct maximal cliques is empty or a one-element set or a line,
lines go to lines in both directions.

Suppose that $f$ and $f^{-1}$ both send  stars to stars.
Then $f$ induces a bijective transformation $f_{k-1}$ of ${\mathcal G}_{k-1}(V)$.
This is an automorphism of $\Gamma_{k-1}(V)$ preserving the types of maximal cliques
and we get an automorphism of the projective space $\Pi_V$ if $k=2$.
In the case when $k\ge 3$, we apply the above arguments to $f_{k-1}$. Step by step, we come to an automorphism of $\Pi_V$.
It is induced by a semilinear automorphism of $V$ (Theorem \ref{theorem-FTPG})
and $f$ is induced by the same semilinear automorphism.

Consider the case when  $f$ transfers a certain star $[S\rangle_{k}$ to a top $\langle U]_{k}$.
Since $f$ and $f^{-1}$ send lines to lines, the restriction of $f$ to this star is an isomorphism 
between the projective spaces $\Pi_{V/S}$ and $\Pi_{U^{*}}$.
Theorem \ref{theorem-FTPG} implies that the vector spaces $V/S$ and $U^{*}$ are of the same dimension
which is possible only in the case when $\dim V$ is finite and equal to $2k$.
Using the intersection properties of maximal cliques, we establish that $f$ transfers every star to a top and 
every top goes to a star.
The composition of $f$ and the annihilator mapping is an isomorphism of $\Gamma_{k}(V)$ to $\Gamma_{k}(V^{*})$
which preserves the types of maximal cliques (the annihilator mapping sends stars to tops and tops to stars).
By the arguments from the previous paragraph, this graph isomorphism is induced by a semilinear isomorphism of $V$ to $V^{*}$.
\end{proof}

\begin{rem}\label{rem-westwick}{\rm
By R. Westwick \cite{West}, every bijective transformation of ${\mathcal G}_{k}(V)$
sending adjacent elements to adjacent elements is an automorphism of the graph $\Gamma_{k}(V)$ if $V$ is finite-dimensional;
in other words, if a bijective transformation of ${\mathcal G}_{k}(V)$ is adjacency preserving in one direction and 
$V$ is finite-dimensional, then this transformation is adjacency preserving in both directions.
Kreuzer's example \cite{Kreuzer} shows that this statement fails for the case when $V$ is infinite-dimensional.
Also, if $V$ is infinite-dimensional, then there is an analogue of Chow's theorem for the Grassmannians ${\mathcal G}_{\infty}(V)$ and ${\mathcal G}^{k}(V)$,
see \cite{Plevnik2}.
}\end{rem}

The diameter of the graph $\Gamma_{k}(V)$, i.e. the maximal path distance between vertices, is equal to $\min\{k, \dim V-k\}$.
We say that two elements of ${\mathcal G}_{k}(V)$ are {\it opposite} if the path distance between them is maximal.
If $\dim V\ge 2k$, then this is equivalent to the fact that the intersection of the subspaces is $0$.
In the case when $\dim V\le 2k$, two elements of ${\mathcal G}_{k}(V)$ are opposite if and only if their sum coincides with $V$.
In the next section, we will use the following.

\begin{theorem}\label{theorem-BH}
If $f$ is a bijective transformation of ${\mathcal G}_{k}(V)$ preserving the relation to be opposite in both directions,
i.e. $X,Y\in {\mathcal G}_{k}(V)$ are opposite if and only if $f(X),f(Y)$ are opposite,
then $f$ is an automorphism of $\Gamma_{k}(V)$.
\end{theorem}

\begin{rem}{\rm
Theorem \ref{theorem-BH} is proved in \cite{HP}  (see also \cite{Pankov-book1}) under the assumption that $V$ is finite-dimensional (the main idea is taken from \cite{BH}),
but the same arguments work in the case when $V$ is of an arbitrary (not necessarily finite) dimension.
The statement is trivial if $k=1$ or $\dim V$ is finite and $k=\dim V -1$.
In the general case, it is a simple consequence of the following characterization of the adjacency relation in terms of the relation to be opposite:
distinct $X,Y\in {\mathcal G}_{k}(V)$ are adjacent if and only if there exists $Z\in {\mathcal G}_{k}(V)\setminus\{X,Y\}$ such that 
every element of ${\mathcal G}_{k}(V)$ opposite to $Z$ is opposite to $X$ or $Y$.
There are more general results concerning transformations preserving pairs with bounded or fixed distance \cite{DeS-VanM, Lim}.
}\end{rem}

\subsection{Apartments preserving transformations}
For every basis $B$ of the vector space $V$ the set formed by all $k$-dimensional subspaces 
spanned by subsets of $B$ is called the {\it apartment} of ${\mathcal G}_{k}(V)$ associated to $B$.
Two bases of $V$ define the same apartment of ${\mathcal G}_{k}(V)$ if and only if 
the vectors from one of the basis are scalar multiples of the vectors from the other.
If $V$ is finite-dimensional, then apartments of ${\mathcal G}_{k}(V)$ are 
the intersections of ${\mathcal G}_{k}(V)$ with apartments of the building ${\mathfrak F}(V)$ (Remark \ref{rem-building}),
every such intersection is a finite set.
In the case when $V$ is infinite-dimensional, every apartment of ${\mathcal G}_{k}(V)$ contains infinitely many elements. 

It is not difficult to prove that for any two subspaces of $V$ there is a basis of $V$ such that each of these subspaces is spanned by a subset of this basis.
As a direct consequence, we get the following remarkable property of apartments:
for any two $k$-dimensional subspaces of $V$ there is an apartment of ${\mathcal G}_{k}(V)$ containing them.

If $\dim V=n$ is finite and  $B=\{e_{i}\}^{n}_{i=1}$ is a basis of $V$,
then the annihilator mapping transfers the associated apartment of ${\mathcal G}_{k}(V)$
to the apartment of ${\mathcal G}_{n-k}(V^{*})$ corresponding to the dual basis $B^{*}$.
The basis $B^{*}$ consists of $e^{*}_{1},\dots,e^{*}_{n}\in V^{*}$ defined by the condition $e^{*}_{i}(e_{j})=\delta_{ij}$,
where $\delta_{ij}$ is the Kronecker delta.

\begin{rem}{\rm
Apartments of Grassmannians have a useful interpretation in terms of exterior products. 
Let us consider the exterior $k$-product $\wedge^{k}V$ 
(which is defined for an arbitrary, not necessarily finite-dimensional, vector space over a field).
This is the vector space (over the same field) whose elements are linear combinations of so-called $k$-{\it vectors} $x_{1}\wedge\dots\wedge x_{k}$, 
where $x_{1},\dots,x_{k}$ are linearly independent vectors from $V$ (see \cite{Sternberg} for the precise definition).
If $\{e_{i}\}_{i\in I}$ is a basis of the vector space $V$, 
then all $k$-vectors of type $e_{i_{1}}\wedge\dots\wedge e_{i_{k}}$, where $i_{1},\dots,i_{k}$ are mutually distinct elements of $I$, 
form a basis of the vector space $\wedge^{k}V$. 
Every such basis of $\wedge^{k}V$ is said to be {\it regular}.
If $\dim V=n$ is finite, then 
$$\dim (\wedge^{k}V)=\binom{n}{k}.$$
If vectors $x_{1},\dots,x_{k}$ and $y_{1},\dots,y_{k}$ span the same $k$-dimensional subspace of $V$,
then 
$$y_{1}\wedge\dots\wedge y_{k}=\det(M)\,x_{1}\wedge\dots\wedge x_{k},$$
where $M$ is the matrix of decomposition of $y_{1},\dots,y_{k}$ in the basis $x_{1},\dots,x_{k}$.
Therefore, the $k$-dimensional subspace of $V$ spanned by vectors $x_{1},\dots,x_{k}$ can be naturally identified with 
the $1$-dimensional subspace of $\wedge^{k}V$ containing the $k$-vector $x_{1}\wedge\dots\wedge x_{k}$.
We get an injective mapping  of ${\mathcal G}_{k}(V)$ to ${\mathcal G}_{1}(\wedge^{k}V)$  whose image consists of all 
$1$-dimensional subspaces of $\wedge^{k}V$ containing $k$-vectors.
This mapping is known as the {\it Pl\"{u}cker embedding}.  It transfers every line of ${\mathcal G}_{k}(V)$ to a line of the projective space $\Pi_{\wedge^{k}V}$.
The apartment of ${\mathcal G}_{k}(V)$ defined by a basis $B$ goes to 
the apartment of ${\mathcal G}_{1}(\wedge^{k}V)$ defined by the regular basis of $\wedge^{k}V$ corresponding to $B$.
}\end{rem}

\begin{rem}{\rm
If $\dim V=n$ is finite, then every apartment of ${\mathcal G}_{k}(V)$ consists of $\binom{n}{k}$ elements and 
it is the image of an  isometric embedding of the Johnson graph $J(n,k)$ in the Grassmann graph $\Gamma_{k}(V)$.
Recall that $J(n,k)$ is the graph whose vertices are $k$-element subsets in a certain $n$-element set  and 
two such subsets are adjacent vertices in the graph if their intersection is a $(k-1)$-element subset.
Note that there are isometric embeddings of $J(n,k)$ in $\Gamma_{k}(V)$ whose images are not apartments \cite[Chapter 4]{Pankov-book2}.
}\end{rem}

The bijective transformations of ${\mathcal G}_{k}(V)$ induced by semilinear automorphisms of $V$ send apartments to apartments. 
If $\dim V=2k$, then the same holds for the transformations of ${\mathcal G}_{k}(V)$ 
defined by semilinar isomorphisms of $V$ to $V^{*}$.

\begin{theorem}[M. Pankov  \cite{Pankov-book1}]\label{theorem-pank2}
If $\dim V\ge 3$ and $f$ is a bijective transformation of ${\mathcal G}_{k}(V)$ such that $f$ and $f^{-1}$ send apartments to apartments,
then $f$ is induced by a semilinear automorphism of $V$ or a semilinear isomorphism of $V$ to $V^{*}$ and the second possibility is realized only 
in the case when $\dim V=2k$.
\end{theorem}

In the case when $k=1$, this statement easily follows from Theorem \ref{theorem-FTPG}.
We observe that three distinct elements of ${\mathcal G}_{1}(V)$ belong to the same apartment if and only if 
they are non-collinear points of the projective space $\Pi_V$, i.e. there is no line of $\Pi_V$ containing them.
Therefore, if $f$ is a bijective transformation of ${\mathcal G}_{1}(V)$ satisfying the condition of Theorem \ref{theorem-pank2},
then $f$ and $f^{-1}$ send triples of non-collinear points to triples of non-collinear points.
This means that triples of collinear points go to triples of collinear points in both directions.
Then $f$ and $f^{-1}$ transfer lines to lines, i.e. $f$ is an automorphism of $\Pi_V$.

If $\dim V=n$ is finite, then it is sufficient to prove Theorem \ref{theorem-pank2} only for the case when $k\le n-k$.
Indeed, if $f$ is an apartments preserving  bijective transformation of ${\mathcal G}_{k}(V)$, 
then  $X\to f(X^{0})^{0}$ is a bijective transformation of ${\mathcal G}_{n-k}(V^{*})$ satisfying the same condition. 
The latter transformation is an automorphism of $\Gamma_{n-k}(V^{*})$ if and only if $f$ is an automorphism of $\Gamma_{k}(V)$.

\begin{rem}{\rm
We refer \cite{Pankov5} for the description of apartments preserving transformations of the Grassmannians of infinite-dimensional vector spaces
formed by subspaces of infinite dimensions and codimensions. 
}\end{rem}

\subsection{Proof of Theorem \ref{theorem-pank2}}
Let $B=\{e_i\}_{i\in I}$ be a basis of $V$. Denote by ${\mathcal A}$ the associated apartment of ${\mathcal G}_{k}(V)$.
We suppose that $k>1$.  In the case when $\dim V=n$ is finite, we also assume that $k\le n-k$.

For every $i\in I$ we denote by ${\mathcal A}(+i)$ and ${\mathcal A}(-i)$
the sets consisting of all elements of ${\mathcal A}$ which contain $e_{i}$ and 
do not contain $e_{i}$, respectively.
For any distinct $i,j\in I$ we define
$${\mathcal A}(+i,+j)={\mathcal A}(+i)\cap {\mathcal A}(+j),$$
$${\mathcal A}(+i,-j)={\mathcal A}(+i)\cap {\mathcal A}(-j).$$
A subset ${\mathcal X}\subset{\mathcal A}$ is said to be {\it inexact}
if there is an apartment of ${\mathcal G}_{k}(V)$ distinct from ${\mathcal A}$ and containing ${\mathcal X}$.

\begin{exmp}\label{exmp-inexact}{\rm
We claim that for any distinct $i,j\in I$ the subset 
\begin{equation}\label{eq4-1}
{\mathcal A}(+i,+j)\cup {\mathcal A}(-i)
\end{equation}
is inexact.
In the basis $B$, we replace $e_{i}$ by the vector $e_{i}+e_{j}$.
If ${\mathcal A}'$ is the apartment of ${\mathcal G}_{k}(V)$ corresponding to this new basis, 
then
$${\mathcal A}\cap {\mathcal A}'={\mathcal A}(+i,+j)\cup {\mathcal A}(-i).$$
}\end{exmp}

\begin{lemma}\label{lemma4-1}
Every maximal inexact subset of ${\mathcal A}$ is of type \eqref{eq4-1}.
\end{lemma}

\begin{proof}
We need to show that every inexact subset ${\mathcal X}\subset {\mathcal A}$
is contained in a subset of type \eqref{eq4-1}.
For every $i\in I$ we denote by $S_{i}$ the intersection of all elements from ${\mathcal X}$ containing $e_i$
and we write $S_{i}=0$ if ${\mathcal X}$ does not contain such elements.
We claim that the dimension of at least one of $S_{i}$ is not equal to $1$
(indeed, if each $S_{i}$ is the $1$-dimensional subspace containing $e_{i}$, then ${\mathcal X}$ is not inexact).
If $S_{i}=0$, then
$${\mathcal X}\subset {\mathcal A}(-i)\subset {\mathcal A}(+i,+j)\cup {\mathcal A}(-i)$$
for any $j\in I\setminus\{i\}$.
In the case when $\dim S_i \ge 2$, we take any $e_{j}\in S_{i}$ such that $j\ne i$ and establish that 
${\mathcal X}$ is contained in ${\mathcal A}(+i,+j)\cup {\mathcal A}(-i)$.
\end{proof}

We say that ${\mathcal C}\subset {\mathcal A}$ is a {\it complementary} subset if 
${\mathcal A}\setminus {\mathcal C}$ is a maximal inexact subset.
An easy verification shows that the complementary subset corresponding to \eqref{eq4-1} is ${\mathcal A}(+i,-j)$.
Our proof is based on the following simple characterization of the relation to be opposite in terms of complementary subsets.

\begin{lemma}\label{lemma4-2}
Two elements of ${\mathcal A}$ are opposite if and only if there is no complementary subset of ${\mathcal A}$ containing both these elements.
\end{lemma}

\begin{proof}
Let $X,Y\in {\mathcal A}$.
The complementary subset ${\mathcal A}(+i,-j)$ contains both $X,Y$ if and only if 
$$e_{i}\in X\cap Y,\;\;e_{j}\not\in X+Y.$$
By our assumption, $2k\le \dim V$, i.e.
two elements of ${\mathcal G}_{k}(V)$ are opposite if and only if their intersection is zero.
If $X\cap Y=0$, then there is no complementary subset containing both $X,Y$. 
Suppose that $X$ and $Y$ have a non-zero intersection and take any $e_{i}\in X\cap Y$.
Since 
$$\dim(X+Y)=2k-\dim(X\cap Y)<2k\le\dim V,$$
there is $e_{j}\not\in X+Y$.
Then the complementary subset ${\mathcal A}(+i,-j)$ contains both $X$ and $Y$.
\end{proof}

Let $f$ be a bijective transformation of ${\mathcal G}_{k}(V)$
such that $f$ and $f^{-1}$ send apartments to apartments.
For any $X,Y\in {\mathcal G}_{k}(V)$ we take an apartment ${\mathcal A}\subset {\mathcal G}_{k}(V)$ containing them.
It is clear that $f$ transfers inexact subsets of ${\mathcal A}$ to inexact subsets of the apartment $f({\mathcal A})$.
Similarly, $f^{-1}$ sends inexact subsets of $f({\mathcal A})$ to inexact subsets of ${\mathcal A}$. 
Therefore, ${\mathcal X}$ is a maximal inexact subset of ${\mathcal A}$
if and only if $f({\mathcal X})$ is a maximal inexact subset of $f({\mathcal A})$.
This means that a subset ${\mathcal C}\subset {\mathcal A}$ is complementary if and only if 
$f({\mathcal C})$ is a complementary subset of $f({\mathcal A})$.
Then Lemma \ref{lemma4-2} guarantees that  $f(X)$ and $f(Y)$ are opposite if and only if the same holds for $X$ and $Y$.
So, $f$ preserves the opposite relation in both directions
and, by Theorem \ref{theorem-BH}, it is an automorphism of the Grassmann graph $\Gamma_{k}(V)$.
Theorem \ref{theorem-chow} gives the claim.

\begin{rem}{\rm
The proof of Theorem \ref{theorem-pank2}  given in \cite{Pankov-book1} 
is based on the following characterization of adjacency in terms of complementary subsets.
Let ${\mathcal A}$ be an apartment of ${\mathcal G}_{k}(V)$.
For any pair of distinct $X,Y\in {\mathcal A}$ we denote by ${\mathfrak C}(X,Y)$ the collection of all complementary subsets of ${\mathcal A}$ containing both $X,Y$. 
An easy verification shows that the following assertions are fulfilled:
\begin{enumerate}
\item[(1)]
In the case when $V$ is finite-dimensional, 
$X,Y\in {\mathcal A}$ are adjacent if and only if ${\mathfrak C}(X,Y)$ contains the maximal number of complimentary subsets.
This number is equal to $(k-1)(\dim V-k-1)$.
\item[(2)]
Suppose that $V$ is infinite-dimensional. 
For any pair of distinct $X,Y\in {\mathcal A}$ there are adjacent $X',Y'\in {\mathcal A}$ such that ${\mathfrak C}(X,Y)$ is contained in  ${\mathfrak C}(X',Y')$;
moreover, if ${\mathfrak C}(X,Y)={\mathfrak C}(X',Y')$, then the pairs $X,Y$ and $X',Y'$ are coincident, i.e. $X,Y$ are adjacent.
\end{enumerate}
If $f$ is a bijective transformation of ${\mathcal G}_{k}(V)$ such that $f$ and $f^{-1}$ send apartments to apartments,
then the statements (1) and (2) guarantee that $f$ is an automorphism of $\Gamma_{k}(V)$.
}\end{rem}

\begin{rem}\label{rem5-finite}{\rm
Suppose that $\dim V=n$ is finite and not less than $3$.
Let $f$ be a bijective transformation of ${\mathcal G}_{k}(V)$ sending apartments to apartments
(we do not require that $f^{-1}$ satisfies the same condition).
It was noted above that we can restrict ourself to the case when $k\le n-k$.
If $k=1$, then $f$ transfers any triple of non-collinear points of $\Pi_{V}$ to a triple of non-collinear points.
This implies that $f^{-1}$ sends any triple of collinear points to a triple of collinear points.
Then $f^{-1}$ maps lines to subsets of lines and, by Remark \ref{rem2-1}, it is an automorphism of $\Pi_{V}$. 
Consider the case when $k>1$.
Let ${\mathcal A}$ be an apartment of ${\mathcal G}_{k}(V)$.
Then $f$ transfers every inexact subset of ${\mathcal A}$ to an inexact subset of the apartment $f({\mathcal A})$.
Since ${\mathcal A}$ and $f({\mathcal A})$ have the same finite number of  inexact subsets,
${\mathcal X}$ is an inexact subset of ${\mathcal A}$ if and only if $f({\mathcal X})$ is an inexact subset of $f({\mathcal A})$.
As above, we establish that $f$ is an automorphism of $\Gamma_{k}(V)$.
Therefore, the statement of Theorem \ref{theorem-pank2} holds even if we do not require that the inverse transformation $f^{-1}$ sends apartments to apartments.
A more general result can be found in \cite[Chapter 5]{Pankov-book2}.
}\end{rem}

\section{Grassmannians of Hilbert spaces}
We return to Grassmannians of Hilbert spaces.
As above, we suppose that $H$ is a complex Hilbert space.
The following two types of bijective transformations of ${\mathcal G}_{k}(H)$ will be considered:
transformations preserving the orthogonality relation in both directions and compatibility preserving transformations.

\subsection{Orthogonality preserving transformations}

By Theorem \ref{theorem3-3} and Proposition \ref{prop3-1}, every bijective transformation of ${\mathcal G}_{\infty}(H)$ or ${\mathcal G}_{1}(H)$
preserving the orthogonality relation in both directions can be uniquely extended to an automorphism of the logic ${\mathcal L}(H)$.
Using Theorem \ref{theorem-chow}, we prove the following.

\begin{theorem}[M. Gy\"ory \cite{Gyory}, P. {\v S}emrl \cite{Semrl}]\label{theorem-GS}
If $\dim H >2k$, then every bijective transformation of ${\mathcal G}_{k}(H)$ preserving the orthogonality relation in both directions 
can be uniquely extended to an automorphism of the logic ${\mathcal L}(H)$.
\end{theorem}

If $\dim H<2k$, then  there exist no orthogonal pairs of $k$-dimensional subspaces.
Consider the case when $\dim H=2k$.
For every $X\in {\mathcal G}_{k}(H)$ the orthogonal complement $X^{\perp}$ is the unique $k$-dimensional subspace orthogonal to $X$.
It was noted in Subsection 3.2 that any bijective transformation of the set of all such pairs $\{X,X^{\perp}\}$ defines a class
of bijective transformations of ${\mathcal G}_{k}(H)$ preserving the orthogonality relation in both directions.
If $\dim H=2$, then every such transformation of ${\mathcal G}_{1}(H)$ can be uniquely extended to a logic automorphism
and there are logic automorphisms which are not induced by unitary and anti-unitary operators.
If $\dim H=2k\ge 4$, then every logic automorphism  is induced by an unitary or anti-unitary operator
and there are bijective transformations of ${\mathcal G}_{k}(H)$ which preserve the orthogonality relation in both directions
and cannot be extended to logic automorphisms.

\begin{proof}[Proof of Theorem \ref{theorem-GS}]
Let $f$ be a bijective transformation of ${\mathcal G}_{k}(H)$ preserving the orthogonality relation in both directions
and $\dim H>2k$. For $k=1$ the statement was proved above (Proposition \ref{prop3-1}) and we suppose that $k>1$.

For every subspace $U\subset H$ we denote by $\langle U]_{k}$ the set of all $k$-dimensional subspaces contained in $U$.
If $U\in {\mathcal G}^{k}(H)$, then $\langle U]_{k}$ consists of all $k$-dimensional subspaces orthogonal to $U^{\perp}\in {\mathcal G}_{k}(H)$.
Hence $f(\langle U]_k)$ is formed by all $k$-dimensional subspaces orthogonal to $f(U^{\perp})$
which means that
$$f(\langle U]_k)=\langle g(U)]_{k},\;\mbox{ where }\;g(U)=f(U^{\perp})^{\perp}$$
and $g$ is a bijective transformation of ${\mathcal G}^{k}(H)$.
The condition $\dim H>2k$ guarantees that ${\mathcal G}^{k}(H)\ne {\mathcal G}_{k}(H)$.

Let ${\mathcal G}'$ be the set of all closed subspaces of $H$ whose codimension is a finite number not less than $k$ and whose dimension is greater than $k$.
If $H$ is infinite-dimensional, then ${\mathcal G}'$ is formed by all closed subspaces of finite codimension $\ge k$.
In the case when $\dim H=n$ is finite, it consists of all subspaces $X\subset H$ satisfying $$k<\dim X\le n-k$$
(this set is non-empty, since $n>2k$).
The set $\langle U]_k$ is non-empty if $U$ belongs to ${\mathcal G}'$ and
we claim that there exists $g(U)\in {\mathcal G}'$ such that 
\begin{equation}\label{eq5-1}
f(\langle U]_k)=\langle g(U)]_{k}.
\end{equation}
Indeed, $U$ can be presented as the intersection of some $U_{1},\dots,U_{i}\in {\mathcal G}^{k}(H)$
and 
$$g(U)=g(U_1)\cap\dots\cap g(U_i)$$
is as required.
We get a bijective transformation $g$ of ${\mathcal G}'$ satisfying \eqref{eq5-1} for every $U\in {\mathcal G}'$.
It is easy to see that $g$ preserves the inclusion relation in both directions.
Hence, it preserves the codimensions of subspaces. 

Let $X$ and $Y$ be elements of ${\mathcal G}_{k}(H)$ such that $(X+Y)^{\perp}$ belongs to ${\mathcal G}'$
(since $\dim H>2k$, any two adjacent elements of ${\mathcal G}_{k}(H)$ satisfy this condition). 
Then $\langle (X+Y)^{\perp}]_{k}$ consists of all $k$-dimensional subspaces orthogonal to both $X,Y$ and
$$f(\langle (X+Y)^{\perp}]_{k})=\langle g((X+Y)^{\perp})]_k$$
is formed by all $k$-dimensional subspaces orthogonal to both $f(X),f(Y)$.
This implies that
$$g((X+Y)^{\perp})=(f(X)+f(Y))^{\perp}.$$
Therefore, 
$$(X+Y)^{\perp}\;\mbox{ and }\;(f(X)+f(Y))^{\perp}$$ are of the same finite codimension.
Since $X,Y$ are adjacent if and only if the codimension of $(X+Y)^{\perp}$ is equal to $k+1$,
the transformation $f$ sends adjacent elements to adjacent elements.
Applying the same arguments to $f^{-1}$, we establish that
$f$ is an automorphism of the Grassmann graph $\Gamma_{k}(H)$.
It follows from Theorem \ref{theorem-chow} that  $f$ is induced by a semilinear automorphism of $H$.
This semilinear automorphism transfers orthogonal vectors to orthogonal vectors and, by Lemma \ref{lemma3-0},
it is a scalar multiple of an unitary or ant-iunitary operator.
\end{proof}

\begin{rem}{\rm
The original proofs from \cite{Gyory,Semrl} are not related to Chow's theorem.
Other proof based on Chow's theorem can be found in \cite{GeherSemrl}.
}\end{rem}

\subsection{Compatibility preserving transformations}
In Section 3.3, we describe bijective transformations of ${\mathcal L}(H)$ and ${\mathcal G}_{\infty}(H)$
preserving the compatibility relation in both directions.
Now, we investigate such kind of transformations for the Grassmannian ${\mathcal G}_{k}(H)$
and  restrict ourself to the case when $H$ is infinite-dimensional. 
Some remarks concerning the finite-dimensional case will be given at the end of this section.

\begin{theorem}[M. Pankov \cite{Pankov4}]\label{theorem-pank3}
If $H$ is infinite-dimensional, then every bijective transformation of ${\mathcal G}_{k}(H)$ 
preserving the compatibility relation in both directions can be uniquely extended to an automorphism of the logic ${\mathcal L}(H)$.
\end{theorem}

For $k=1$ this statement is a simple consequence of Proposition \ref{prop3-1}.
Indeed, two distinct elements of ${\mathcal G}_{1}(H)$ are compatible if and only if they are orthogonal.

In the case when $k>1$, the proof will be based on the notion of orthogonal apartment.
For every orthogonal basis $B$ of $H$ the set of all $k$-dimensional subspaces spanned by subsets of $B$
is said to be the {\it orthogonal apartment} of ${\mathcal G}_{k}(H)$ associated to $B$, in other words,
this is the intersection of ${\mathcal A}(B)$ and the Grassmannian ${\mathcal G}_{k}(H)$.

Recall that a subset of ${\mathcal L}(H)$ is called {\it compatible} if any two distinct elements of this subset are compatible.
By Proposition \ref{prop3-3}, orthogonal apartments can be characterized as maximal compatible subsets of ${\mathcal G}_{k}(H)$,
i.e. the family of orthogonal apartments of ${\mathcal G}_{k}(H)$ coincides with the family of maximal compatible subsets of ${\mathcal G}_{k}(H)$.
Therefore, for a bijective transformation $f$ of ${\mathcal G}_{k}(H)$ the following two conditions are equivalent:
\begin{enumerate}
\item[$\bullet$] $f$ and $f^{-1}$ send orthogonal apartments to orthogonal apartments,
\item[$\bullet$] $f$ preserves the compatibility relation in both directions.
\end{enumerate}
We will use some modifications of the arguments from Section 4.3 to show that 
every bijective transformation $f$ of ${\mathcal G}_{k}(H)$ satisfying the above conditions is 
orthogonality preserving in both directions.

We need to explain why the method from Section 3.4 cannot be exploited to study compatibility preserving transformations of ${\mathcal G}_{k}(H)$.
As above, we suppose that $H$ is infinite-dimensional. If $k$ is odd, then
$$\{X,Y\}^{cc}\cap {\mathcal G}_{k}(H)=\{X,Y\}$$
for any distinct compatible $X,Y\in {\mathcal G}_{k}(H)$.
If $k$ is even, then the same is true, except the case when $\dim(X\cap Y)=k/2$.
The latter equality implies that the subspace 
$$(X\cap Y^{\perp})+(Y\cap X^{\perp})$$
is $k$-dimensional, i.e. it belongs to $\{X,Y\}^{cc}\cap {\mathcal G}_{k}(H)$.

\subsection{Proof of Theorem \ref{theorem-pank3}}
Let $\{e_{i}\}_{i\in I}$ be an orthogonal basis of $H$
and let ${\mathcal A}$ be the associated orthogonal apartment of ${\mathcal G}_{k}(H)$. 
We suppose that $H$ is infinite-dimensional and $k>1$. 

As in Section 4.3,
for every $i\in I$ we denote by ${\mathcal A}(+i)$ and ${\mathcal A}(-i)$
the sets consisting of all elements of ${\mathcal A}$ which contain $e_{i}$ and 
do not contain $e_{i}$, respectively.
For any distinct $i,j\in I$ we define
$${\mathcal A}(+i,+j)={\mathcal A}(+i)\cap {\mathcal A}(+j),$$
$${\mathcal A}(+i,-j)={\mathcal A}(+i)\cap {\mathcal A}(-j),$$
$${\mathcal A}(-i,-j)={\mathcal A}(-i)\cap {\mathcal A}(-j).$$
A subset ${\mathcal X}\subset{\mathcal A}$ is said to be {\it orthogonally inexact}
if there is an orthogonal apartment of ${\mathcal G}_{k}(H)$ distinct from ${\mathcal A}$ and containing ${\mathcal X}$.

\begin{exmp}{\rm
We state that for any distinct $i,j\in I$ the subset 
\begin{equation}\label{eq5-2}
{\mathcal A}(+i,+j)\cup {\mathcal A}(-i,-j)
\end{equation}
is orthogonally inexact. 
In the basis $\{e_{i}\}_{i\in I}$, we replace the vectors $e_{i}$ and $e_{j}$
by any other pair of orthogonal vectors belonging to 
the $2$-dimensional subspace spanned by $e_{i}$ and $e_{j}$ 
(these new vectors are not scalar multiples of $e_i$ and $e_j$).
If ${\mathcal A}'$ is the associated orthogonal apartment of ${\mathcal G}_{k}(H)$ , then
$${\mathcal A}\cap {\mathcal A}'={\mathcal A}(+i,+j)\cup {\mathcal A}(-i,-j).$$
This means that \eqref{eq5-2} is orthogonally inexact.
}\end{exmp}

\begin{lemma}\label{lemma5-1}
Every maximal orthogonally inexact subset in ${\mathcal A}$ is of type \eqref{eq5-2}.
\end{lemma}

\begin{proof}
We need to show that every orthogonally inexact subset ${\mathcal X}\subset {\mathcal A}$
is contained in a subset of type \eqref{eq5-2}.
For every $i\in I$ we denote by $S_{i}$ the intersection of all subspaces $X$ 
satisfying one of the following conditions:
\begin{enumerate}
\item[(1)]  
$X$ is an element of ${\mathcal X}$ containing $e_{i}$,
\item[(2)]  
$X$ is the orthogonal complement of an element from ${\mathcal X}$ which does not contain $e_{i}$. 
\end{enumerate}
Each $S_i$ is non-zero. 
If ${\mathcal A}'$ is the orthogonal apartment defined by an orthogonal basis $\{e'_{i}\}_{i\in I}$ 
and ${\mathcal X}$ is contained in ${\mathcal A}'$, 
then every subspace $X$ satisfying (1) or (2) and, consequently every $S_{i}$, is spanned by a subset of $\{e'_{i}\}_{i\in I}$.
Therefore, if every $S_{i}$ is $1$-dimensional, then 
${\mathcal A}$ is the unique orthogonal apartment containing ${\mathcal X}$
which contradicts the fact that ${\mathcal X}$ is orthogonally inexact.
So, there is at least one $i\in I$ such that $\dim S_{i}\ge 2$.
We take any $j\ne i$ such that $e_{j}$ belongs to $S_{i}$ and claim that
$${\mathcal X}\subset {\mathcal A}(+i,+j)\cup {\mathcal A}(-i,-j).$$
If $X\in {\mathcal X}$ contains $e_{i}$, then $e_{j}\in S_{i}\subset X$
and $X$ belongs to ${\mathcal A}(+i,+j)$. 
If $X\in {\mathcal X}$ does not contain $e_{i}$, then $e_{j}\in S_{i}\subset X^{\perp}$
which means that $e_{j}$ is not contained in $X$
and $X$ belongs to ${\mathcal A}(-i,-j)$. 
\end{proof}

We say that ${\mathcal C}\subset {\mathcal A}$ is an {\it orthocomplementary} subset
if ${\mathcal A}\setminus {\mathcal C}$ is a maximal orthogonally inexact subset, i.e.
$${\mathcal A}\setminus {\mathcal C}={\mathcal A}(+i,+j)\cup {\mathcal A}(-i,-j)$$
for some distinct $i,j\in I$.
The latter equality implies that 
$${\mathcal C}={\mathcal A}(+i,-j)\cup {\mathcal A}(+j,-i).$$
This orthocomplementary subset will be denoted by ${\mathcal C}_{ij}$.
Note that ${\mathcal C}_{ij}={\mathcal C}_{ji}$.

In the case when $H$ is infinite-dimensional, 
there is a simple characterization of orthogonality in terms of orthocomplementary subsets.

\begin{lemma}\label{lemma5-2}
Suppose that $H$ is infinite-dimensional.
Then $X,Y\in{\mathcal A}$ are orthogonal if and only if the number of orthocomplementary subsets of ${\mathcal A}$ 
containing both $X$ and $Y$ is finite.
\end{lemma}

\begin{proof}
If the orthocomplementary subset ${\mathcal C}_{ij}$ contains both $X$ and $Y$, then 
one of the following possibilities is realized:
\begin{enumerate}
\item[(1)] one of $e_{i},e_{j}$ belongs to $X\setminus Y$ and 
the other to $Y\setminus X$,
\item[(2)] one of $e_{i},e_{j}$ belongs to $X\cap Y$ and the other is not contained in $X+Y$.
\end{enumerate}
The number of orthocomplementary subsets ${\mathcal C}_{ij}$ satisfying (1) is finite. 
If $X$ and $Y$ are orthogonal, then $X\cap Y=0$ and there is no ${\mathcal C}_{ij}$ satisfying (2).
In the case when $X\cap Y\ne 0$, the condition (2) holds for infinitely many ${\mathcal C}_{ij}$. 
\end{proof}

Let $f$ be a bijective transformation of ${\mathcal G}_{k}(H)$ preserving the compatibility relation in both directions,
in other words, $f$ and $f^{-1}$ send orthogonal apartments to orthogonal apartments.
For any orthogonal $k$-dimensional subspaces $X,Y\subset H$
there is an orthogonal apartment ${\mathcal A}\subset {\mathcal G}_{k}(H)$ containing them.
It is clear that $f$ sends orthogonally inexact subsets of ${\mathcal A}$ 
to orthogonally inexact subsets of the orthogonal apartment $f({\mathcal A})$.
Similarly, $f^{-1}$ transfers orthogonally inexact subsets of $f({\mathcal A})$ to orthogonally inexact subsets of ${\mathcal A}$. 
This means that  ${\mathcal X}$ is a maximal orthogonally inexact subset of ${\mathcal A}$
if and only if $f({\mathcal X})$ is a maximal orthogonally inexact subset of $f({\mathcal A})$.
Therefore, a subset ${\mathcal C}\subset {\mathcal A}$ is orthocomplementary if and only if 
$f({\mathcal C})$ is an orthocomplementary subset of $f({\mathcal A})$.
Lemma \ref{lemma5-2} guarantees  that  $f(X)$ and $f(Y)$ are orthogonal.
Similarly, we establish that $f^{-1}$ transfers orthogonal elements to orthogonal elements.  
So, $f$ preserves the orthogonality relation in both directions and we apply Theorem \ref{theorem-GS}.

\subsection{Compatibility preserving transformations. The finite-dimensional case}
If $H$ is finite-dimensional, then all orthogonal apartments (maximal compatible subsets) of ${\mathcal G}_{k}(H)$
have the same finite number of elements.
This implies that for a bijective transformation $f$ of ${\mathcal G}_{k}(H)$ the following two conditions are equivalent:
\begin{enumerate}
\item[$\bullet$] $f$ sends orthogonal apartments to orthogonal apartments,
\item[$\bullet$] $f$ sends compatible elements to compatible elements.
\end{enumerate}
Under the assumption that $\dim H$ is finite and not equal to $2k$ 
we can characterize pairs of adjacent elements in orthogonal  apartments by orthocomplementary subsets for almost all cases. 
This is impossible only in the case when $\dim H=6$ and $k\in \{2,4\}$. 
For example, if $\dim H=6$, then two distinct compatible elements of ${\mathcal G}_{2}(H)$ are adjacent or orthogonal
and the distinguishing of such two possibilities is an open problem.
Using the mentioned above characterization and  Westwick's generalization of Chow's theorem (see Remark \ref{rem-westwick}) we establish the following.

\begin{theorem}[M. Pankov \cite{Pankov4}]\label{theorem-pank4}
Let $f$ be a bijective transformation of ${\mathcal G}_{k}(H)$ sending compatible elements to compatible elements
$($we do not assume that the same holds for the inverse transformation $f^{-1}$$)$.
Suppose that $\dim H$ is finite and not equal to $2k$.
In the case when $\dim H=6$, we also require that $k$ is distinct from $2$ and $4$.
Then $f$ can be uniquely extended to an automorphism of the logic ${\mathcal L}(H)$.
\end{theorem}

In the case when $\dim H=2k$, we get the following result similar to Theorems \ref{theorem-MS} and \ref{theorem-plevnik}.

\begin{theorem}[M. Pankov \cite{Pankov4}]\label{theorem-pank5}
Suppose that $\dim H=2k\ge 8$ and $f$ is a bijective transformation of ${\mathcal G}_{k}(H)$ 
preserving the compatibility relation in both directions.
There exists an automorphism $g$ of the logic ${\mathcal L}(H)$ such that 
for every $X\in {\mathcal G}_{k}(H)$ we have either 
$$f(X)=g(X)\;\mbox{ or }\;f(X)=g(X)^{\perp}.$$
\end{theorem}

The proofs of the above statements involve some technical details 
and the arguments do not work for the case when $\dim H$ is equal to $4$ or $6$.
For this reason, we do not present them here.

\subsection*{Acknowledgment}
The author expresses his deep gratitude to Antonio Pasini for useful remarks and interesting discussions.

\end{document}